\documentclass[oneside]{amsart}
\usepackage[utf8]{inputenc}
\usepackage[english]{babel}
\usepackage{multicol}
\usepackage{stmaryrd}
\usepackage{amsmath, amsthm}
\usepackage{amssymb}
\usepackage{amsfonts}
\usepackage[a4paper]{geometry}
\usepackage{systeme}
\usepackage{hyperref}
\usepackage{xcolor}
\usepackage{fancyhdr}
\usepackage{dsfont}
\usepackage{graphicx}
\usepackage{float}
\usepackage{pgfplots}
\pgfplotsset{compat=1.18}
\definecolor{mplblue}{HTML}{1F77B4}
\definecolor{run1}{HTML}{FF0000}
\definecolor{run2}{HTML}{FF6A00}
\definecolor{run3}{HTML}{B200FF}
\definecolor{run4}{HTML}{FF006E}
\definecolor{run5}{HTML}{008C4D}
\definecolor{run6}{HTML}{4800FF}
\definecolor{run7}{HTML}{9B0046}
\definecolor{newred}{HTML}{BF1418}

\newtheorem{theorem}{Theorem}[section]
\newtheorem{lemma}[theorem]{Lemma}
\newtheorem{proposition}[theorem]{Proposition}

\newtheorem{corollary}[theorem]{Corollary}

\newtheorem{definition}[theorem]{Definition}

\theoremstyle{definition} 
\newtheorem{example}[theorem]{Example}

\newcommand\norm[1]{\left\lVert#1\right\rVert}
\newcommand{\R}{\mathbb{R}}

\newcommand{\N}{\mathbb{N}}
\newcommand{\E}{\mathbb{E}}
\newcommand{\I}{\mathbb{I}}
\newcommand{\ii}{\mathbf{i}}
\newcommand{\jj}{\mathbf{j}}
\newcommand{\pp}{\mathbf{p}}
\newcommand{\Geo}{\mathrm{Geo}}
\newcommand{\Proba}{\mathbb{P}}

\newcommand\std{\mathrm{std}}
\newcommand\maj{\mathrm{maj}}
\def\restriction#1#2{\mathchoice
              {\setbox1\hbox{${\displaystyle #1}_{\scriptstyle #2}$}
              \restrictionaux{#1}{#2}}
              {\setbox1\hbox{${\textstyle #1}_{\scriptstyle #2}$}
              \restrictionaux{#1}{#2}}
              {\setbox1\hbox{${\scriptstyle #1}_{\scriptscriptstyle #2}$}
              \restrictionaux{#1}{#2}}
              {\setbox1\hbox{${\scriptscriptstyle #1}_{\scriptscriptstyle #2}$}
              \restrictionaux{#1}{#2}}}
\def\restrictionaux#1#2{{#1\,\smash{\vrule height .8\ht1 depth .85\dp1}}_{\,#2}} 
\renewcommand{\epsilon}{\varepsilon}

\begin{document}
\title{Cycle structure of random standardized permutations}
\author{Aurélien Guerder}
\address{IECL – Site de Nancy, Faculté des sciences et Technologies, Campus, Boulevard des Aiguillettes, 54506 Vandœuvre-lès-Nancy, France}
\email{aurelien.guerder@univ-lorraine.fr}

\begin{abstract}
In this article, we study a model of random permutations, which we call random standardized permutations, based on a sequence of i.i.d.~random variables. This model generalizes others, such as the riffle-shuffle and the major-index-biased permutations. We first establish an exact result on the joint distribution of the number of cycles of given lengths, involving the notion of primitive words. From this result, we obtain various convergence results, most of which are proved using the method of moments. First we prove that the number of small cycles may have either a Poisson limit distribution, or a limit distribution given by a countable sum of independent geometric distributions. Then we establish a limit distribution for large cycles, which is the Poisson-Dirichlet process. Finally we prove a central limit theorem for the total number of cycles.
\end{abstract}

\maketitle
\fancypagestyle{myheadings}{%
    \fancyhf{} %
    \fancyhead[CE,CO]{\small\scshape %
        \ifodd\value{page} Cycle structure of random standardized permutations \else Aurélien Guerder \fi
    }
    \fancyfoot[C]{\thepage}
    \renewcommand{\headrulewidth}{0pt} %
}

\pagestyle{myheadings}

\section{Introduction}

\subsection{Background and model} The study of random permutations is a central subject in probability theory and combinatorics. Many aspects spark interest in the field, such as inversions, descents, cycles, records,  or longest increasing subsequences. The first natural model is the uniform one, for which many results are already established (see e.g.~\cite{ArratiaBarbourTavare} for cycle structure). More recent articles have shifted their attention toward non-uniform models, such as riffle-shuffle (see \cite{BayerDiaconis}), Mallows permutations (see \cite{arXiv:2410.17228}) or record biased permutations (see \cite{arXiv:2409.01692}). Despite this, the uniform model remains a useful reference for comparison, and it is interesting to observe that some results are identical in some non-uniform models.

Given a sequence $g = (g_1, \dots, g_n) \in \R^n$, the standardized permutation of $g$, denoted by $\std(g)$, is the permutation $\sigma \in \mathfrak{S}_n$ defined as follows: if the smallest value of the $g_j, j \leq n$ is $g_{j_1}= \dots =g_{j_l}=x_0$ with $j_1< \dots <j_l$, we have $\sigma(j_1)=1, \dots ,\sigma(j_l)=l$, then we continue counting with the second smallest value, and so on. For example, if $g = (6,1,5,3,3,1,2)$, we get $\std(g) = 7164523 $. To get a \textit{random permutation}, we choose $G$ randomly. Here we focus on the case where that the $G_j$ are i.i.d.~random variables. Since the atomless case produces a uniform random permutation (see Section \ref{casuniforme}), we assume in all this article that the $G_j$ take values in a countable set $\I$, and, for $i \in \I$, we write $p_i = \Proba(G_j = i)$. This model generalizes some others, such as the riffle-shuffle if the $G_j$ are uniform in a finite set (see \cite{BayerDiaconis} or \cite{StanleyQS}), or major-index-biased permutations if the $G_j$ follow a geometric distribution (see \cite{StanleyQS} or \cite{arXiv:2501.12513}).

\subsection{Overview of the results}
\subsubsection{Exact distribution of cycles}
We first get an exact and explicit result about the distribution of cycles, by generalizing some arguments of \cite{arXiv:2501.12513}. Since $\std$ is not injective, we need to study the sequence $G$ which generated the permutation $\std(G)$, in order to avoid loss of information. We break down the problem according to the values of the $G_j$ variables. For a given $\ii \in \I^k$, we study the number $D_{\ii}$ of $k$-cycles $(x_1,\dots,x_k)$ such that $G_{x_1}=i_1, \dots , G_{x_k} = i_k$. Recall that a word is primitive if it is not a power of a smaller word. It appears that the appropriate k-tuples to consider, to avoid counting a given cycle more than once, are actually the conjugacy classes of primitive words of length $k$ (see Section \ref{primitivewords}), constituting the set $\widetilde{Q}_k$. We then have that, for any $k \geq 1$, the number $c_k$ of cycles of length $k$ of $\std(G)$ is given by $$c_k = \sum_{\ii \in \widetilde{Q}_k} D_{\ii}.$$ Define $p_{\ii} = p_{i_1}\dots p_{i_k}$ for a given $\ii \in \I^k$, and $Q_k$ the set of all primitive words of $\I^k$. Then, our first main result is the following: 

\begin{theorem}\label{loiDiCycles}
For any $k_1, \dots ,k_r \geq 1, \ii_1 \in Q_{k_1}, \dots , \ii_r \in Q_{k_r} $ pairwise non conjugate, and $l_1, \dots ,l_r \geq 0$,
$$ \Proba\left( D_{\ii_1} \geq l_1, \dots , D_{\ii_r} \geq l_r  \right) = \begin{cases}
p_{\ii_1}^{l_1} \dots p_{\ii_r}^{l_r} & \text{if } k_1l_1 +  \dots  + k_rl_r \leq n;\\
0 & \text{else.}
\end{cases} $$
\end{theorem}

This theorem characterizes the joint distribution of the $D_{\ii}$. This result generalizes previous findings in \cite{diaconiscycledescents} concerning riffle-shuffle cycles, as well as work in \cite{arXiv:2501.12513} concerning fixed points and 2-cycles of major-index biased permutations. To prove this result (Section \ref{sectioncycles}), we use a bijection between $\I^n$ and $\left\{ g' \in \I^{n+k} \mid D_{\ii}(g') \geq 1 \right\}$, which adds a cycle of type $\ii \in Q_k$ to $\std(g)$ without changing the rest of the permutation.

\subsubsection{Asymptotic behavior of short cycles}

We define the "short cycles" as those of length smaller than or equal to a fixed $k$, as $n$ tends to infinity. We note $\Geo_0(p)$ the geometric distribution starting from 0, i.e.~if $X$ follows $\Geo_0(p)$ and $k \in \N$, then $ \Proba(X = k) = (1-p)^kp$. If we let $n$ tend to infinity without letting the distribution of the $G_j$ depend on $n$, the previous theorem immediately gives us the convergence of the finite-dimensional distributions of the sequence $ (D_{\ii})_{\ii \in \widetilde{Q}_l, l \leq k}$ to those of $$\underset{i \in \widetilde{Q}_l, l \leq k}{\bigotimes} \Geo_0(1-p_{\ii}).$$ Since $c_k$ is the countable sum of all the $D_{\ii}, \ii \in \widetilde{Q}_k$, to get the convergence in distribution of $c_k$, we prove a convergence result in the set $\ell^1$:

\begin{theorem}\label{CVloiDiloifixe}
Fix $k \geq 1$. Then, $$(D_{\ii})_{\ii \in \widetilde{Q}_l, l \leq k} \overset{(d)}{\longrightarrow} \underset{i \in \widetilde{Q}_l, l \leq k}{\bigotimes} \Geo_0(1-p_{\ii})$$ in $\displaystyle \ell^1\left(\bigcup_{l \leq k}\widetilde{Q}_l\right)$ as $n \rightarrow +\infty$.
\end{theorem}

\begin{corollary}\label{CVloipetitscyclesloifixe}
Denote by $c_j$ the number of cycles of $\std(G)$ of length $j$. We have $$\displaystyle \left( c_1 ,\dots , c_k  \right) \overset{(d)}{\longrightarrow} \underset{l \leq k}{\bigotimes} \left( \sum_{\ii \in \widetilde{Q}_l} X_{\ii}\right) $$ where $ \left( X_{\ii} \right)_{\ii \in \widetilde{Q}_l, l\leq k}$ follows the distribution $\underset{\ii \in \widetilde{Q}_l, l \leq k}{\bigotimes} \Geo_0(1-p_{\ii})$.
\end{corollary}

Under a different hypothesis on the distribution $\pp$, which consists on assuming that it is spreading asymptotically, we get the following result, which has already been obtained in \cite[Proposition 5.4]{diaconiscycledescents} in the riffle-shuffle model:

\begin{theorem}\label{CyclesLoiVariable}
Assume that the distribution of the $G_k^{(n)}$ satisfies $$\norm{\pp^{(n)}}_{\infty} = \max_{i \in \I} p_i^{(n)}  \underset{n \rightarrow +\infty}{\longrightarrow} 0. $$ 
Then, for $k \in \N^*$,  $$ \left(c_1, \dots ,c_k \right)  \overset{(d)}{\underset{n \rightarrow +\infty}{\longrightarrow}} \mathcal{P} \left( 1 \right) \otimes  \dots  \otimes \mathcal{P} \left(  \frac{1}{k} \right) . $$
\end{theorem}

\subsubsection{Asymptotic behavior of large cycles}
We define the infinite-dimensional simplex $$\displaystyle \Delta^{\infty} = \left\{ x \in [0,1]^{\N^*} \;\big|\; x_1 \geq x_2 \geq  \dots  \text{ and } \sum_{j \in \N^*} x_j \leq 1 \right\}, $$ endowed with the topology of pointwise convergence. The study of the infinite-dimensional vector $n^{-1}\lambda^{(n)} = n^{-1}(\lambda_1^{(n)},\lambda_2^{(n)}, \dots )$ of $ \Delta^{\infty}$, where $\lambda_j^{(n)}$ is the size of the $j$-th greatest cycle of $\std(G)$, can be seen as an asymptotic definition of large cycles, since the short cycles, say the fixed points, appear at the end of the vector, and hence have no influence on the pointwise convergence of $n^{-1}\lambda^{(n)}$. Then, under a mild assumption on the distribution $\pp$, the vector converges to the Poisson-Dirichlet process with parameter 1:

\begin{theorem}\label{CVPDQS}

Let $R \in (0;1)$, and suppose that for any $n \geq 1$ and $i \in \I$, we have $p_i^{(n)} \leq R$. Then, the vector $\dfrac{\lambda^{(n)}}{n} = n^{-1}(\lambda_1^{(n)},\lambda_2^{(n)}, \dots )$, where $\lambda_j^{(n)}$ is the size of the $j$-th greatest cycle of $\std(G)$, converges in distribution in $\Delta^{\infty}$ to the Poisson-Dirichlet process with parameter 1.
\end{theorem}
To prove this theorem, we use the result for the uniform model, and we transfer the limit distribution thanks to a criterion of convergence in distribution in the set $\Delta^{\infty}$. More precisely, the convergence in distribution in $\Delta^{\infty}$ is characterized by the expectations of joint moments of functions of Poisson-Dirichlet process (see Theorem \ref{CVloisimplexe}). This value is hard to compute directly, but it is obtained by using the convergence in the uniform model. Note that this result has also already been proved in \cite[Proposition 5.5]{diaconiscycledescents} in the riffle-shuffle model.

\subsubsection{Asymptotic behavior of cycle count}

The last statistic that we study in this article is the total number of cycles, which is asymptotically normal (see Section \ref{Sectioncyclecountnormality}).

\begin{theorem}{\label{Cyclecountnormality}}
Let $R \in (0;1)$, and suppose that for any $n \geq 1$ and $i \in \I$, we have $p_i^{(n)} \leq R$. Then, as $ n \rightarrow +\infty$, $$ \frac{K_n - \log(n)}{\sqrt{\log(n)}} \overset{(d)}{\longrightarrow} \mathcal{N}(0,1)$$ and the moments converge.
\end{theorem}

This theorem seems to be new even in the riffle-shuffle case. The proof of the theorem consists on studying the convergence of cumulants. It appears that studying the number of cycles of length smaller than $n/\log(\log(n))$ is more convenient and is equivalent -- this is reminiscent of the proof Erdös-Kac theorem (see \cite{ErdosKac}, or \cite[Theorem 30.3]{BillPandM} for a more recent reference), where it is more convenient to study the number of prime divisors smaller than $n^{1/\log(\log(\log(n)))}$ of a random integer, rather than the total number of divisors. Combinatorics of cumulants is also required, through the Leonov-Shiryaev formula (see Theorem \ref{LeonovShiryaev}).

\section{Random standardized permutations model}

\subsection{Definition}

Throughout this paper, we use one-line notation for permutations, i.e.~for any $\sigma \in \mathfrak{S}_n$, we write $\sigma = \sigma_1 \sigma_2 \dots \sigma_n$. If $\sigma$ is the $k$-cycle mapping $x_1$ to $x_2$, $x_2$ to $x_3$, \dots, $x_k$ to $x_1$, we write $\sigma = (x_1, x_2, \dots, x_k)$. 

\begin{definition}
Let $g = (g_1, \dots ,g_n)\in \R^n$. We define the permutation $\sigma := \std(g) \in \mathfrak{S}_n$ by $$ \sigma(j) < \sigma(k) \iff g_j < g_k \text{ or } \left\{ g_j = g_k \text{ and } j<k \right\}   $$ called standardized permutation of $g$.

\end{definition}

In other words, if the smallest value of $\left\{ g_k, k \leq n \right\}$ is $g_{k_1}= \dots =g_{k_l}=x_0$ with $k_1< \dots <k_l$, we have $\sigma(k_1)=1, \dots ,\sigma(k_l)=l$, then we continue labeling with the second smallest value, and so on. For example, if $g = (6,1,5,3,3,1,2)$, we get $\std(g) = 7164523 $ (see Figure \ref{IllustrAlgo}). In this paper, we consider the case where $g$ is random. More precisely, we consider $G = (G_1, \dots ,G_n)$, where the $G_k$ are $n$ i.i.d.~random variables  with distribution $\mu$. If the distribution of the $G_j$ is atomless, the permutation $\std(G)$ is uniform (see Section \ref{casuniforme}). Throughout this text, we assume that $\mu$ is a discrete distribution on $\R$. Specifically, we fix a countable subset $\I \subset \R$ and we assume that $\mu$ has support in $\I$. The distribution $\mu$ may depend on $n$ or not, when we study asymptotic behaviors. We note $p_i = \mu(\{i\}) = \Proba(G_k=i)$ for any $i \in \I$ (possibly $p_i^{(n)}$ if the distribution depends on $n$), and $\pp = (p_i)_{i \in \I}$ (possibly $\pp^{(n)}$).

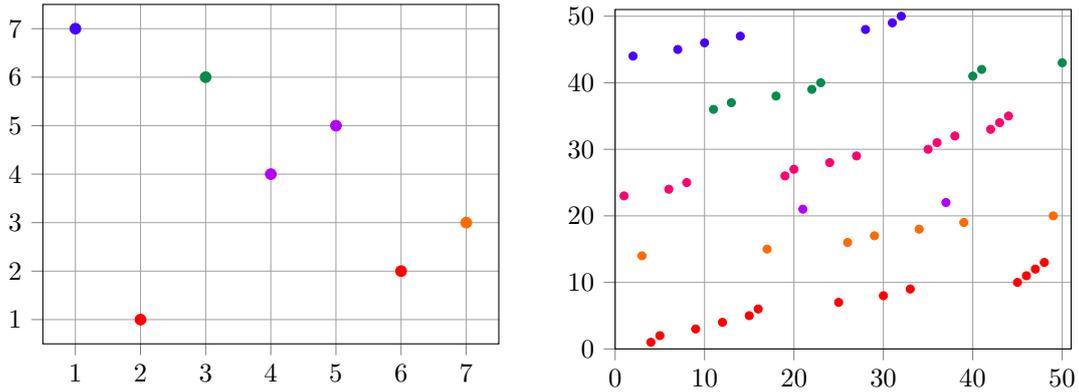
\begin{figure}[!h]
\centering
\begin{minipage}[c]{.48\linewidth}
    \centering
	\begin{tikzpicture}
    \begin{axis}[
        scale only axis,  
        width=6cm,        %
        height=4.5cm,       %
        xmin=0.5, xmax=7.5,
        ymin=0.5, ymax=7.5,
        xtick={1,2,3,4,5,6,7},
        ytick={1,2,3,4,5,6,7},
        grid=both,
        major grid style={gray!70},
        axis on top=false,
        tick align=outside,
        tick pos=left
    ]

    \addplot[
        only marks,       
        mark=*,           %
        mark size=2pt,  %
        color=run1     %
    ] coordinates {
        (2, 1)
        (6, 2)
    };
\addplot[
        only marks,       
        mark=*,           %
        mark size=2pt,  %
        color=run2     %
    ] coordinates {
        (7, 3)
    };
\addplot[
        only marks,       
        mark=*,           %
        mark size=2pt,  %
        color=run3     %
    ] coordinates {
        (4, 4)
        (5, 5)
    };
\addplot[
        only marks,       
        mark=*,           %
        mark size=2pt,  %
        color=run5     %
    ] coordinates {
        (3, 6)
    };
\addplot[
        only marks,       
        mark=*,           %
        mark size=2pt,  %
        color=run6     %
    ] coordinates {
        (1, 7)
    };

    \end{axis}
	\end{tikzpicture}
    \end{minipage}
    \hfill%
    \begin{minipage}[c]{.48\linewidth}
    \centering
	\begin{tikzpicture}
    \begin{axis}[
        scale only axis,  
        width=6cm,        %
        height=4.5cm,       %
        xmin=0, xmax=51,
        ymin=0, ymax=51,
        xtick={0,10,20,30,40,50},
        ytick={0,10,20,30,40,50},
        grid=both,
        major grid style={gray!70},
        axis on top=false,
        tick align=outside,
        tick pos=left
    ]

    \addplot[
        only marks,       
        mark=*,           %
        mark size=1.5pt,  %
        color=run1     %
    ] coordinates {
	(4,1)
	(5,2)
	(9,3)
	(12,4)
	(15,5)
	(16,6)
	(25,7)
	(30,8)
	(33,9)
	(45,10)
	(46,11)
	(47,12)
	(48,13)
    };
	\addplot[
        only marks,       
        mark=*,           %
        mark size=1.5pt,  %
        color=run2     %
    ] coordinates {
    (3,14)
    (17,15)
    (26,16)
    (29,17)
    (34,18)
    (39,19)
    (49,20)
    };
	\addplot[
        only marks,       
        mark=*,           %
        mark size=1.5pt,  %
        color=run3     %
    ] coordinates {
    (21,21)
    (37,22)
    };
	\addplot[
        only marks,       
        mark=*,           %
        mark size=1.5pt,  %
        color=run4     %
    ] coordinates {
    (1,23)
    (6,24)
    (8,25)
    (19,26)
	(20,27)
	(24,28)
	(27,29)
	(35,30)
	(36,31)
	(38,32)
	(42,33)
	(43,34)
	(44,35)
    };
	\addplot[
        only marks,       
        mark=*,           %
        mark size=1.5pt,  %
        color=run5     %
    ] coordinates {
    (11,36)
    (13,37)
    (18,38)
	(22,39)
	(23,40)
	(40,41)
	(41,42)
	(50,43)
    };
	\addplot[
        only marks,       
        mark=*,           %
        mark size=1.5pt,  %
        color=run6     %
    ] coordinates {
    (2,44)
    (7,45)
    (10,46)
    (14,47)
    (28,48)
    (31,49)
	(32,50)
    };
    \end{axis}
	\end{tikzpicture}
    \end{minipage}
\caption{Standardized permutation of $g = (6,1,5,3,3,1,2)$ and of a sequence of length 50 taking values in $\llbracket 1,6 \rrbracket $, where the colors denote the points whose x-coordinates are $k$ such that $g_k = 1$, then 2, up to $6$ (from bottom to top). Note that we plot the $i$ on the horizontal axis and the $\sigma(i)$ on the vertical axis}\label{IllustrAlgo}
\end{figure}

\subsection{Some particular cases} The aim of this short part is to link our model to other well-known ones, that it generalizes.
\subsubsection{Atomless variables and uniform model}\label{casuniforme}
In this part only, we assume that the measure $\mu$ has no atoms, i.e.~for any $x \in \R, \mu(\{x\}) = 0$. In this case, it holds that:

\begin{proposition}
If $\mu$ has no atoms, then $\std(g)$ is a uniform random permutation on $\mathfrak{S}_n$.
\end{proposition}

\begin{proof}
Since $\mu$ is atomless, the $G_k$ are pairwise distinct almost surely: we consider the event $A = \{G_k \text{ pairwise distinct} \} $. Note that for any $\sigma \in \mathfrak{S}_n$, $$ A \cap (\std(G) = \sigma) \iff \left( G_{\sigma^{-1}(j)} \right)_{1 \leq j \leq n} \text{ is increasing.} $$
Since the $G_k$ are i.i.d.~random variables, for any $\sigma \in \mathfrak{S}_n$, $(G_1,\dots ,G_n)$ and $\left( G_{\sigma^{-1}(1)}, \dots ,G_{\sigma^{-1}(n)} \right) $ have the same distribution. Therefore, since $\Proba(A) = 1$, the quantity 
\begin{multline*}
\Proba\left( \std(G) = \sigma \right) = \Proba\left(  (\std(G) = \sigma) \cap A \right) \\= \Proba\left( \left( G_{\sigma^{-1}(j)} \right)_{1 \leq j \leq n} \text{ is increasing} \right) = \Proba\left( \left( G_j \right)_{1 \leq j \leq n} \text{ is increasing} \right)
\end{multline*}
does not depend on $\sigma$, from which the result follows.
\end{proof}

\subsubsection{Major-index-biased permutations}

In \cite{Strahov}, \cite{feraymeliot}, or \cite{arXiv:2501.12513}, the authors deal with major-index-biased permutations. For a given $\sigma \in \mathfrak{S}_n$, the major index $\maj(\sigma)$ of $\sigma$ is the sum of its descents, i.e.$$ \maj(\sigma) = \sum_{\sigma(i) > \sigma(i+1)} i. $$
Fix $q > 0$, then a random permutation $\sigma \in \mathfrak{S}_n$ is a major-index-biased permutation if it is selected with probability proportional to $q^{\maj(\sigma)}$ (see Figure \ref{IllustrMajInd} for an example).\\
Coopman gives a random permutation generating algorithm, which appears at first sight to be different from standardized permutations, but which turns out to be equivalent. 

\begin{definition}\cite{arXiv:2501.12513} For $g = (g_1, \dots ,g_n)\in \R^n$, let $\preceq$ be an order on the pairs $(k, g_k)$ defined by $$ (k,g_k) \preceq (l, g_l) \iff g_k > g_l \text{ or }  \left\{ g_k = g_l \text{ et } k<l \right\}. $$
By ordering the pairs in increasing order and ignoring the first element of each pair, we get a permutation which we denote by $\Gamma(g)$.
\end{definition}

A relationship exists between $\std$ and $\Gamma$, given by$$ \Gamma(-g) = \std(g) ^{-1}$$ for any $g = (g_1, \dots ,g_n)\in \R^n$. We then get this particular case of random standardized permutation, which follows from \cite[Theorem 1.1]{arXiv:2501.12513}, or from \cite[Theorem 2.1]{StanleyQS}:

\begin{proposition}
Let $q \in (0,1)$, and $G = (G_1, \dots, G_n)$ be a sequence of i.i.d.~geometric random variables with parameter $1-q$. Then, $\std(-G)^{-1}$ is a major-index-biased permutation with parameter $q$.
\end{proposition}

\begin{figure}[!h]
\begin{tikzpicture}
    \begin{axis}[
        scale only axis,  
        width=6cm,        %
        height=5cm,       %
        xmin=0, xmax=1005,
        ymin=0, ymax=1005,
        xtick={0,200,400,600,800,1000},
        ytick={0,200,400,600,800,1000},
        grid=both,
        major grid style={gray!70},
        axis on top=false,
        tick align=outside,
        tick pos=left
    ]

    \addplot[
        only marks,       
        mark=*,           %
        mark size=0.8pt,  %
        color=mplblue     %
    ] table [x=x, y=y] {pts_majind.txt};
    \end{axis}
	\end{tikzpicture}
\centering
\caption{Example of a major-index-biased permutation with $n = 1000$ and $q = 0.7$}\label{IllustrMajInd}
\end{figure}
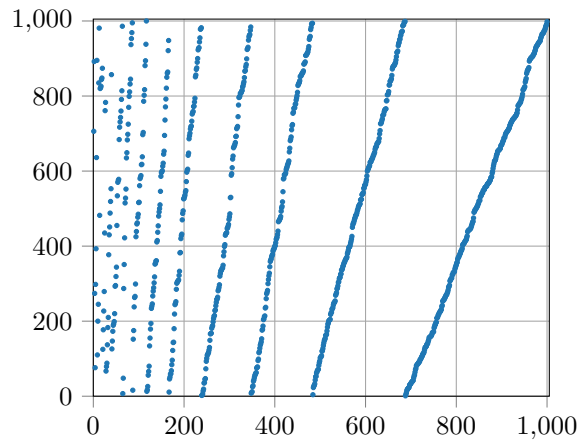

Note that, for symmetry reasons, we may restrict the study of major-index-biased distribution to the case $q \in (0,1)$: if $q >1$, a $q$-biased permutation with parameter $q$ can be obtained by taking the height complement of a $1/q$-biased permutation.

Some results are known for this model, in particular concerning increasing subsequences \cite{arXiv:2501.12513, feraymeliot}, or more recently regarding patterns, fixed points and 2-cycles \cite{arXiv:2501.12513}.

\subsubsection{Riffle-Shuffle}

The $q$-riffle-shuffle was originally introduced to model card shuffling: the 2-riffle-shuffle of a deck of cards consists in splitting the deck into two at a randomly chosen position (according to a binomial distribution), and then merging these two subdecks (see Figure \ref{riffleshuffle}, or \cite[Figure 1, Figure 2]{BayerDiaconis}). More generally, we obtain a $q$-riffle-shuffle by splitting the deck into $q$ subdecks (see for example \cite{BayerDiaconis} or \cite{fulman} for more details).

\begin{figure}[!h]
\begin{tikzpicture}
    \begin{axis}[
        scale only axis,  
        width=6cm,        %
        height=5cm,       %
        xmin=0, xmax=53,
        ymin=0, ymax=53,
        xtick={0,10,20,30,40,50},
        ytick={0,10,20,30,40,50},
        grid=both,
        major grid style={gray!70},
        axis on top=false,
        tick align=outside,
        tick pos=left
    ]

    \addplot[
        only marks,       
        mark=*,           %
        mark size=1.5pt,  %
        color=mplblue     %
    ] table [x=x, y=y] {pts_riffleshuffle.txt};

    \end{axis}
	\end{tikzpicture}
\centering
\caption{A 2-riffle-shuffle of a deck of 52 cards}
\label{riffleshuffle}
\end{figure}
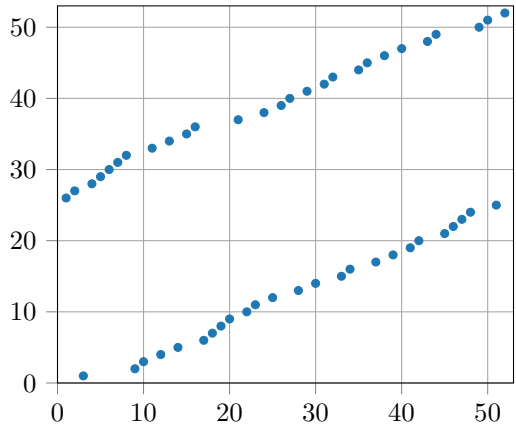 

The standardized permutation model generalizes the $q$-riffle-shuffle. We indeed have the following property, noted by Stanley in \cite{StanleyQS}:

\begin{proposition}
Let $q \geq 1$. The distribution of $\std(G)$, where the $G_k$ are independent uniform variables on $ \llbracket 1,q \rrbracket $, is the $q$-riffle-shuffle.
\end{proposition}

Various results are known for this model, such as the number of cycles of a given length (involving the notion of primitive words), and convergence of small and large cycles \cite{diaconiscycledescents}. These results are generalized in Sections \ref{sectionloicycles}, \ref{sectionpetitscycles} and \ref{sectiongdscycles}. We have not found any previous reference to the total number of cycles for this model.

\subsection{Runs of a standardized permutation}

Let $g \in \I^n$. For $i \in \I$, we define the sets of indices $\mathcal{G}_i = \left\{k \in \llbracket 1,n \rrbracket \mid g_k = i\right\}$, and $\mathcal{G}_{< i} = \left\{k \in \llbracket 1,n \rrbracket \mid g_k < i\right\}$. We also define the "runs" $L_i : \llbracket 0,n \rrbracket \to \llbracket 0,n \rrbracket $ of $\std(g)$ by\\
1. $L_i(0) = \left| \mathcal{G}_{< i} \right|$\\
2. for $k = 1, \dots ,n,$ $ L_i(k) = \begin{cases}
L_i(k-1) & \text{ if } g_k \neq i\\
L_i(k-1)+1 & \text{ if } g_k = i.
\end{cases}$\\
If we consider the runs of $\std(g)$ and $\std(g')$ for two given sequences $g$ and $g'$, we will write $L_i^{g}$ and $L_i^{g'}$. Note that the $L_i^g$ function has at least one fixed point, as $L_i^g(0)\geq 0, L_i^g(n)\leq 0$ and $L_i^g(x)$ increases by 0 or 1 when $x$ increases by 1. 

The $L_i$ are interpolations of the points of $\sigma$ associated to the $g_k = i$, generalizing those introduced in \cite{arXiv:2501.12513}. Observe that, if $\sigma = \std(G)$, and $g_k = i$, then we have $L_i(k) = \sigma(k)$ and $L_i(k-1) = \sigma(k)-1$, and conversely, if $L_i(k) = L_i(k-1)+1$, then $\sigma(k) = L_i(k)$. Thus, the points of $\sigma$ correspond to the up-steps of the $L_i$ (see Figure \ref{IllustrRuns}). Note that the word "runs" usually refers to the successive rises of a permutation. Here, they can be viewed as (parts of) the usual runs of the inverse permutation.

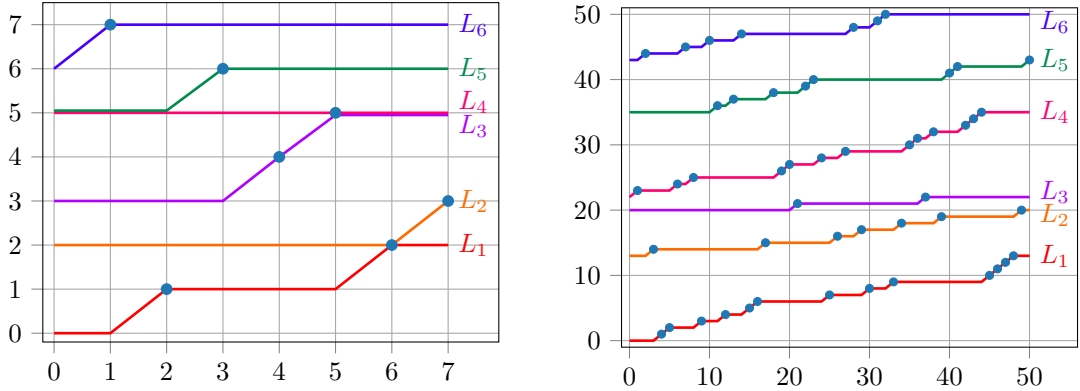
\begin{figure}[!h]
\centering
\begin{minipage}[c]{.48\linewidth}
        \centering
\begin{tikzpicture}
    \begin{axis}[
        scale only axis,  
        width=6cm,        %
        height=4.5cm,       %
        xmin=-0.2, xmax=7.9,
        ymin=-0.2, ymax=7.5,
        xtick={0,1,2,3,4,5,6,7},
        ytick={0,1,2,3,4,5,6,7},
        grid=both,
        major grid style={gray!70},
        axis on top=false,
        tick align=outside,
        tick pos=left
    ]

    \addplot[
        color=run1,      %
        line width=1pt,  %
        no markers     
    ] coordinates {
    	(0, 0)
    	(1, 0)
        (2, 1)
        (5, 1)
        (6, 2)
        (7, 2)
    }
    node[right, pos=1] {$L_1$};
    \addplot[
        only marks,       
        mark=*,           %
        mark size=2pt,  %
        color=mplblue     %
    ] coordinates {
        (2, 1)
        (6, 2)
    };
    \addplot[
        color=run2,      %
        line width=1pt,  %
        no markers     
    ] coordinates {
    	(0, 2)
    	(6, 2)
    	(7, 3)
    }
    node[right, pos=1] {$L_2$};
	\addplot[
        only marks,       
        mark=*,           %
        mark size=2pt,  %
        color=mplblue     %
    ] coordinates {
        (7, 3)
    };
    \addplot[
        color=run3,      %
        line width=1pt,  %
        no markers     
    ] coordinates {
    	(0, 3)
    	(3, 3)
    	(4, 4)
        (5, 4.95)
        (7, 4.95)
    }
    node[pos=1,right,yshift=-4pt] {$L_3$};
	\addplot[
        only marks,       
        mark=*,           %
        mark size=2pt,  %
        color=mplblue     %
    ] coordinates {
        (4, 4)
        (5, 5)
    };
    \addplot[
        color=run4,      %
        line width=1pt,  %
        no markers     
    ] coordinates {
    	(0, 5)
        (7, 5)
    }
    node[pos=1,right,yshift=4pt] {$L_4$};
    \addplot[
        color=run5,      %
        line width=1pt,  %
        no markers     
    ] coordinates {
    	(0, 5.05)
    	(2, 5.05)
    	(3, 6)
        (7, 6)
    }
    node[pos=1,right,yshift=0pt] {$L_5$};
	\addplot[
        only marks,       
        mark=*,           %
        mark size=2pt,  %
        color=mplblue     %
    ] coordinates {
        (3, 6)
    };
    \addplot[
        color=run6,      %
        line width=1pt,  %
        no markers     
    ] coordinates {
    	(0, 6)
    	(1, 7)
    	(7, 7)
    }
    node[pos=1,right,yshift=0pt] {$L_6$};
	\addplot[
        only marks,       
        mark=*,           %
        mark size=2pt,  %
        color=mplblue     %
    ] coordinates {
        (1, 7)
    };

    \end{axis}
	\end{tikzpicture}        
    \end{minipage}
    \hfill%
    \begin{minipage}[c]{.48\linewidth}
        \centering
        \begin{tikzpicture}
    \begin{axis}[
        scale only axis,  
        width=6cm,        %
        height=4.5cm,       %
        xmin=-1, xmax=56,
        ymin=-1, ymax=51,
        xtick={0,10,20,30,40,50},
        ytick={0,10,20,30,40,50},
        grid=both,
        major grid style={gray!70},
        axis on top=false,
        tick align=outside,
        tick pos=left
    ]

    \addplot[
        color=run1,      %
        line width=1pt,  %
        no markers     
    ] coordinates {
    (0,0)
    (3,0)
    (4,1)
	(5,2)
	(8,2)
	(9,3)
	(11,3)
	(12,4)
	(14,4)
	(15,5)
	(16,6)
	(24,6)
	(25,7)
	(29,7)
	(30,8)
	(32,8)
	(33,9)
	(44,9)
	(45,10)
	(46,11)
	(47,12)
	(48,13)
	(50,13)
    }
    node[pos=1,right,yshift=0pt] {$L_1$};
    \addplot[
        only marks,       
        mark=*,           %
        mark size=1.5pt,  %
        color=mplblue     %
    ] coordinates {
	(4,1)
	(5,2)
	(9,3)
	(12,4)
	(15,5)
	(16,6)
	(25,7)
	(30,8)
	(33,9)
	(45,10)
	(46,11)
	(47,12)
	(48,13)
    };
    \addplot[
        color=run2,      %
        line width=1pt,  %
        no markers     
    ] coordinates {
    (0,13)
    (2,13)
    (3,14)
    (16,14)
    (17,15)
    (25,15)
    (26,16)
    (28,16)
    (29,17)
    (33,17)
    (34,18)
    (38,18)
    (39,19)
    (48,19)
    (49,20)
    (50,20)
    }
    node[pos=1,right,yshift=-2pt] {$L_2$};
	\addplot[
        only marks,       
        mark=*,           %
        mark size=1.5pt,  %
        color=mplblue     %
    ] coordinates {
    (3,14)
    (17,15)
    (26,16)
    (29,17)
    (34,18)
    (39,19)
    (49,20)
    };
    \addplot[
        color=run3,      %
        line width=1pt,  %
        no markers     
    ] coordinates {
    (0,20)
    (20,20)
    (21,21)
    (36,21)
    (37,22)
    (50,22)
    }
    node[pos=1,right,yshift=2pt] {$L_3$};
	\addplot[
        only marks,       
        mark=*,           %
        mark size=1.5pt,  %
        color=mplblue     %
    ] coordinates {
    (21,21)
    (37,22)
    };
    \addplot[
        color=run4,      %
        line width=1pt,  %
        no markers    
    ] coordinates {
    (0,22)
    (1,23)
    (5,23)
    (6,24)
    (7,24)
    (8,25)
    (18,25)
    (19,26)
    (19,26)
	(20,27)
	(23,27)
	(24,28)
	(26,28)
	(27,29)
	(34,29)
	(35,30)
	(36,31)
	(37,31)
	(38,32)
	(41,32)
	(42,33)
	(43,34)
	(44,35)
	(50,35)
    }
    node[pos=1,right,yshift=0pt] {$L_4$};
	\addplot[
        only marks,       
        mark=*,           %
        mark size=1.5pt,  %
        color=mplblue     %
    ] coordinates {
    (1,23)
    (6,24)
    (8,25)
    (19,26)
	(20,27)
	(24,28)
	(27,29)
	(35,30)
	(36,31)
	(38,32)
	(42,33)
	(43,34)
	(44,35)
    };
    \addplot[
        color=run5,      %
        line width=1pt,  %
        no markers     
    ] coordinates {
    (0,35)
    (10,35)
    (11,36)
    (12,36)
    (13,37)
    (17,37)
    (18,38)
    (21,38)
	(22,39)
	(23,40)
	(39,40)
	(40,41)
	(41,42)
	(49,42)
	(50,43)
    }
    node[pos=1,right,yshift=0pt] {$L_5$};
	\addplot[
        only marks,       
        mark=*,           %
        mark size=1.5pt,  %
        color=mplblue     %
    ] coordinates {
    (11,36)
    (13,37)
    (18,38)
	(22,39)
	(23,40)
	(40,41)
	(41,42)
	(50,43)
    };
    \addplot[
        color=run6,      %
        line width=1pt,  %
        no markers     
    ] coordinates {
    (0,43)
    (1,43)
    (2,44)
    (6,44)
    (7,45)
    (9,45)
    (10,46)
    (13,46)
    (14,47)
    (27,47)
    (28,48)
    (30,48)
    (31,49)
	(32,50)
	(50,50)
    }
    node[pos=1,right,yshift=-3pt] {$L_6$};
	\addplot[
        only marks,       
        mark=*,           %
        mark size=1.5pt,  %
        color=mplblue     %
    ] coordinates {
    (2,44)
    (7,45)
    (10,46)
    (14,47)
    (28,48)
    (31,49)
	(32,50)
    };
    \end{axis}
	\end{tikzpicture}
    \end{minipage}
\caption{Example on runs with $g = (6,1,5,3,3,1,2)$ and with the sequence of Figure \ref{IllustrAlgo}}\label{IllustrRuns}
\end{figure}

\section{Exact distribution of cycle count}\label{sectioncycles}

\subsection{Distribution of fixed point count}

For $i \in \I$, denote  $\mathcal{D}_i = \mathcal{D}_i(G) = \left\{k \in \mathcal{G}_i | \sigma(k) = k\right\}$ for $\sigma = \std(G)$. That is, $\mathcal{D}_i$ is the set of the fixed points of $\std(G)$ associated to $G_k = i$. We also have $\mathcal{D}_i = \{ x \in \llbracket 1, n \rrbracket \mid L_i(x) = x , L_i(x-1) = x-1  \} $; in other words, since the fixed points of the $L_i$ function are consecutive, $\mathcal{D}_i$ is the set of all the fixed points of $L_i$ excluding the first. Graphically, an element of $\mathcal{D}_i$ corresponds to a rise of 1 of $L_i$ on the line $y=x$ (see Figure \ref{IllustrRunsPtsfixes}). We also write $D_i = |\mathcal{D}_i|$ the number of fixed points of $\std(G)$ associated to the value $i$.

\begin{figure}[!h]
\begin{tikzpicture}
    \begin{axis}[
        scale only axis,  
        width=7cm,        %
        height=5cm,       %
        xmin=-0.2, xmax=7.7,
        ymin=-0.2, ymax=7.4,
        xtick={0,1,2,3,4,5,6,7},
        ytick={0,1,2,3,4,5,6,7},
        grid=both,
        major grid style={gray!70},
        axis on top=false,
        tick align=outside,
        tick pos=left
    ]

    \draw[dashed ] (0 ,0 ) -- ( 4,4 );
    \draw[dashed ] (5,5 ) -- ( 7,7 );
    \addplot[
        color=run1,      %
        line width=1pt,  %
        no markers     %
    ] coordinates {
    	(0, 0)
    	(1, 1)
        (5, 1)
        (6, 2)
        (7, 2)
    }
    node[right, pos=1] {$L_1$};
    \addplot[
        only marks,       
        mark=*,           %
        mark size=2pt,  %
        color=mplblue     %
    ] coordinates {
        (1, 1)
        (6, 2)
    };
    \addplot[
        color=run2,      %
        line width=1pt,  %
        no markers     %
    ] coordinates {
    	(0, 2)
    	(6, 2)
    	(7, 3)
    }
    node[right, pos=1] {$L_2$};
	\addplot[
        only marks,       
        mark=*,           %
        mark size=2pt,  %
        color=mplblue     %
    ] coordinates {
        (7, 3)
    };
    \addplot[
        color=run3,      %
        line width=1pt,  %
        no markers     %
    ] coordinates {
    	(0, 3)
    	(3, 3)
    	(4, 4)
        (5, 4.98)
        (7, 4.98)
    }
    node[pos=1,right,yshift=-4pt] {$L_3$};
	\addplot[
        only marks,       
        mark=*,           %
        mark size=2pt,  %
        color=mplblue     %
    ] coordinates {
        (4, 4)
        (5, 5)
    };
    \addplot[
        color=run4,      %
        line width=1pt,  %
        no markers     %
    ] coordinates {
    	(0, 5.02)
        (7, 5.02)
    }
    node[pos=1,right,yshift=4pt] {$L_4$};
    \addplot[
        color=run5,      %
        line width=1pt,  %
        no markers     %
    ] coordinates {
    	(0, 5.06)
    	(2, 5.06)
    	(3, 6)
        (7, 6)
    }
    node[pos=1,right,yshift=0pt] {$L_5$};
	\addplot[
        only marks,       
        mark=*,           %
        mark size=2pt,  %
        color=mplblue     %
    ] coordinates {
        (3, 6)
    };
    \addplot[
        color=run6,      %
        line width=1pt,  %
        no markers     %
    ] coordinates {
    	(0, 6)
    	(1, 6)
    	(2, 7)
    	(7, 7)
    }
    node[pos=1,right,yshift=0pt] {$L_6$};
	\addplot[
        only marks,       
        mark=*,           %
        mark size=2pt,  %
        color=mplblue     %
    ] coordinates {
        (2, 7)
    };

    \end{axis}
	\end{tikzpicture}   
\centering
\caption{Example with $g = (1,6,5,3,3,1,2)$, here $D_1 = 1$ (and 1 is a fixed point of type $1$ of $\std(g)$, and $D_3 = 2$ (and 4 and 5 are fixed points of type $3$ of $\std(g)$).}
\label{IllustrRunsPtsfixes}
\end{figure}
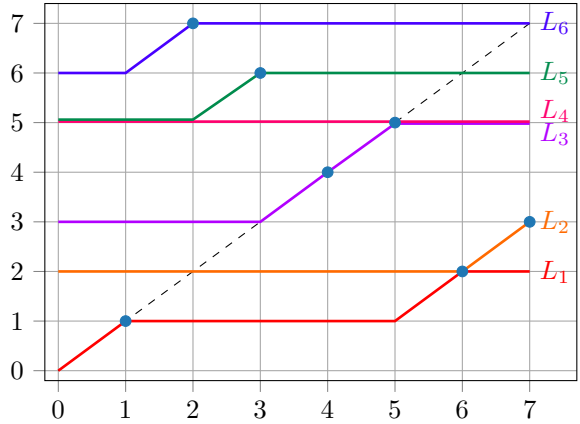

Let us denote by $\Geo_0(p)$ the geometric distribution with parameter $p$ and starting from 0, i.e.~if $X$ follows the distribution $\Geo_0(p)$, then for any $k \in \N, \Proba(X=k) = (1-p)^kp$. Recall that $\Proba(X \geq k) = (1-p)^k$ and that its expectation is $\E(X) = (1-p)/p$.

\begin{proposition}\label{loiDi}
For any $i \in \I, D_i$ follows the distribution $\min(\Geo_0(1-p_i),n)$.
\end{proposition}

\begin{proof}
Fix $i \in \I$ and $0 \leq l \leq n$, and write $A = \{ g \in \I^n \mid D_i(g) \geq l \}$. For any $a \in \llbracket 0, n-l \rrbracket$, we write $B_a = \left\{ g \in \I^n \mid \min \{ x \mid L_i^{g}(x) = x \} = a \right\}$ and $B'_a = \left\{ g \in \I^{n-l} \mid \min \{ x \mid L_i^{g}(x) = x \} = a \right\}$. Note that $$A = \underset{a \in \llbracket 0, n-l \rrbracket }{\bigsqcup} A \cap B_a  \quad \text{ and } \quad \I^{n-l} = \underset{a \in \llbracket 0, n-l \rrbracket }{\bigsqcup} B'_a.$$
For $a \in \llbracket 0, n-l \rrbracket$, define the following map between $A \cap B_a$ and $B'_a$:

$$\begin{array}{ccccccc}
\phi_a& : & A \cap B_a & \to & B'_a &\\
& & g & \mapsto & g' & \text{where } g'_s = \begin{cases}
g_s & s \leq a;\\
g_{s+l} & s > a.
\end{cases}
\end{array} $$\\
In other words, we delete the first $l$ values of $g$ associated to a fixed point of $L_i^g$, which all equal $i$. The run $L_i^{g'}$ has the same first fixed point as $L_i^{g}$, hence $\phi_a(g)$ is in $B'_a$. We claim that $\phi_a$ is a bijection: if $g' = \phi_a(g)$, we get back $g$ by inserting $l$ values $i$ between $g'_a$ and $g'_{a+1}$. Hence, the inverse of $\phi_a$ is
$$\begin{array}{ccccccc}
\phi_a^{-1} & : & B'_a & \to & A \cap B_a &\\
& & g' & \mapsto & g & \text{where } g_s = \begin{cases}
g'_s & s\leq a;\\
i & a < s \leq a+l;\\
g'_{s-l} & s > a+l.
\end{cases}
\end{array} $$\\ 
The above map takes values in $A \cap B_a$, since the first fixed point of $L_i^g$ is the same as $L_i^{g'}$, and since we inserted $l$ values $i$, $D_i(g) \leq l$. Then, for a given $a \in \llbracket 0, n-l \rrbracket$, since the $G_j$ are i.i.d.,
\begin{align*}
\Proba(G \in A \cap B_a) & = \Proba \left( \restriction{G}{\llbracket a,a+l-1 \rrbracket^c} \in B'_a, G_s = i ~ \forall s \in \llbracket a+1, a+l \rrbracket  \right) = \Proba \left( \restriction{G}{\llbracket 1,n-l \rrbracket} \in B'_a \right) p_i^l.
\end{align*} 
We finally sum the above expression over all the $a \in \llbracket 1,n-l \rrbracket $ to get 
$$\Proba \left( D_i(G) \geq l \right)  =  \Proba\left( G \in A \right)  = p_i^l \sum_{a=0}^{n-l} \Proba \left( \restriction{G}{\llbracket 1,n-l \rrbracket} \in B'_a \right)
 = p_i^l \Proba \left( \restriction{G}{\llbracket 1,n-l \rrbracket} \in \I^{n-l} \right) = p_i^l$$
which concludes the proof. 
\end{proof}

We also have this more general result, describing the dependencies between the $D_i$: 

\begin{proposition}
For $i_1, \dots ,i_r \in \I$ pairwise distinct, and $l_1, \dots ,l_r \geq 0$, we have
$$\Proba(D_{i_1}\geq l_1, \dots ,D_{i_r}\geq l_r)=
\begin{cases}
p_{i_1}^{l_1} \dots p_{i_r}^{l_r} & \text{ if } l_1+ \dots +l_r \leq n\\
0 &\text{ else.}
\end{cases}$$
\end{proposition}

The proof is similar to the one of Proposition \ref{loiDi}. Theorem \ref{loiDiCycles} is a more general result, and is proved in Section \ref{sectionloicycles}.

\begin{corollary}
Denote $c_1$ the number of fixed point of $\std(G)$. Then,
$$\displaystyle \E(c_1) = \sum_{i \geq 0 } \sum_{l=1}^n  p_i^l \, .$$
\end{corollary}

\subsection{Distribution of cycle count}\label{sectionloicycles}

We now consider the number of cycles of our model. First observe that, if we write $G_{x_1} = i_1,  \dots , G_{x_k} = i_k $, and if $(x_1,  \dots ,  x_k)$ is a cycle of $\std(G)$, then $L_{i_1}(x_1) = x_2, L_{i_2}(x_2) = x_3  \dots , L_{i_k}(x_k) = x_1$, i.e. $x_1$ is a fixed point of $L_{i_k}\circ L_{i_{k-1}}\circ  \dots \circ L_{i_1}$, and $x_1-1$ is also a fixed point. In what follows, we study the fixed points of $L_{i_k}\circ L_{i_{k-1}}\circ  \dots \circ L_{i_1}$. Actually, the same methods as for the fixed points will be applied, with some more elements needed to generalize the result. Note that a $k$-cycle $(x_1, \dots, x_k)$ of $\std(G)$ may correspond to fixed points of compositions of runs for multiple $k$-tuples: since $(x_1, \dots, x_k) = (x_2, \dots,  x_k, x_1)$, it also corresponds to fixed points of $L_{i_1}\circ L_{i_{k}}\circ  \dots \circ L_{i_2}$. It also corresponds to fixed points of $L_{i_k}\circ L_{i_{k-1}}\circ  \dots \circ L_{i_1} \circ L_{i_k}\circ L_{i_{k-1}}\circ  \dots \circ L_{i_1}$, i.e.~for a $2k$-tuple, even though the cycle has order $k$ rather than $2k$. To identify which $k$-tuples $\ii = (i_1, \dots ,i_k)$ we have to study, we need the notion of primitive words.

\subsubsection{Primitive words, conjugacy classes}\label{primitivewords}

\begin{definition}
Fix $k \in \N$, and $A$ a set called alphabet ($A$ will be equal to $\I$ in what follows).
\begin{itemize}
\item A word (i.e.~a tuple) $x \in A^k$ is a (nontrivial) power if there exists a word $y$ and an integer $l \geq 2$ such that $x = y^l = yy\dots yy$ ($l$ times). The word $x$ is primitive if it is not a power.
\item We say a word $x$ is a conjugate of a word $y$ if $x$ is a cyclic shift of $y$, that is, there exist two words $u,v$ such that $x = uv$ and $y = vu$. The conjugacy relation is an equivalence relation.
\item We write $Q_k(A)$ the set of primitive words of length $k$ over $A$, and $ \widetilde{Q}_k(A)$ the set of all the conjugacy classes of words of $Q_k(A)$. 
\end{itemize}
\end{definition}

This is a standard notion in language and automaton theory; for more details, see for example \cite{Shallit}. The set $ \widetilde{Q}_k(A)$ is also called the set of all aperiodic necklaces of length $k$ over $A$. Primitive words are, in a sense, the "building blocks" of words, since the following result holds:

\begin{proposition}\cite[Theorem 2.3.4]{Shallit}
Every word $x$ can be written uniquely as a (trivial or nontrivial) power of a primitive word $\mathbf{r}$ called the root of $x$.
\end{proposition}

We now show some remarkable properties about the conjugates of a primitive word.

\begin{proposition}\label{ConjPrimitif1}
Let $x$ be a primitive word. Then, if $x = uv$, where $u$ and $v$ are nonempty, then $y = vu$ is different from $x$.
\end{proposition}

\begin{proof}
Let us proceed by contradiction: suppose that $x = uv = vu$ where $u$ and $v$ are nonempty. Suppose further that $u$ has minimal length. Then, the word $v$ begins with the word $u$: we can write $ v = uv'$. Then, we have $x = u^2 v' = uv'u$. Hence, by removing the initial $u$, we get $v = uv' = v'u$. Repeat as long as $|v'| \geq |u|$ so that $v$ can be written $v = u^l \tilde{v} = \tilde{v}u^l$, where $l \geq 1$ and $|\tilde{v}| < |u|$. For $x$, we then get $x = u^{l+1}\tilde{v} = \tilde{v}u^{l+1}$. Since $u$ has minimal length, $\tilde{v}$ must be the empty word, and $x = u^{l+1}$, which contradicts the fact that $x$ is primitive.
\end{proof}

\begin{proposition}\cite[Theorem 2.4.2]{Shallit}\label{ConjPrimitif2}
Let $w$ and $x$ be conjugated words. Then, $x$ is a nontrivial power if and only if $w$ is a nontrivial power.
\end{proposition}

\begin{proof}
Since $w$ and $x$ are conjugated, there exist two words $u$ and $v$ such that $w = uv$ and $x = vu$. Suppose that $w$ is a nontrivial power, then there exists a word $y$, and $k \geq 2$ such that $w = y^k$. Then, $y^k = uv$. If $|u|$ is a multiple of $|y|$, then $u = y^i$ for an integer $i$, and $v = y^{k-i}$, and hence, $x = vu = y^k$. Else, we write $u = y^ir$ and $v = sy^{k-i-1}$, where $r,s$ are nonempty words such that $y = rs$. We then get $x = vu = sy^{k-i-1}y^ir = s(rs)^{k-1}r = (sr)^k$, which proves that $x$ is a nontrivial power.
\end{proof}

Remark that this proof also gives us that if $w$ and $x$ are conjugated with $w = y^k$, then $x = z^k$ where $y$ and $z$ are conjugated.

From the propositions \ref{ConjPrimitif1} and \ref{ConjPrimitif2}, it follows:

\begin{proposition}\label{conjuguesPrim}
The conjugacy class of a primitive word of length $k$ is a set of $k$ distinct primitive words.
\end{proposition}

\begin{proposition}\label{reverseword}
Let $x = x_1\dots x_k$ be a primitive word. Then, its reverse word $\tilde{x} = x_k\dots x_1$ is a primitive word.
\end{proposition}

\begin{proof}
Suppose that $\tilde{x}$ is a nontrivial power, and write $\tilde{x} = y^l$, for a $l \geq 2$. Then, $x = \tilde{y}^l$ where $\tilde{y}$ is the reverse word of $y$, and hence $x$ is a nontrivial power.
\end{proof}

In what follows, we only consider words over the alphabet $\I$, we then note $Q_k$ instead of $Q_k(\I)$. This notion is essential for the cycle combinatorics of $\std(G)$, but will also complicate some arguments. While the notion of primitive words is needed to obtain exact results, it is often irrelevant in asymptotic computations. We have for example the following result:

\begin{proposition}\label{motsPrimNegl}
If the distribution $(p_i)_{i\in \I}$ on $\I$ is non-degenerate, then 
$$  \Proba\left((G_1G_2 \dots G_k) \in Q_k \right) \underset{k \rightarrow +\infty}{\longrightarrow} 1 $$
and more generally, for $l \geq 1$, $ \displaystyle \sum_{\ii \in Q_k} (p_{i_1} \dots p_{i_k})^l \underset{k \rightarrow +\infty}{\sim} \sum_{\ii \in \I^k} (p_{i_1} \dots p_{i_k})^l = \left( \sum_{i \in \I} p_i^l \right)^k = \norm{\pp}_l^{lk}$.
\end{proposition}

\begin{proof}
Fix $l \geq 1$.
\begin{align*}
\norm{\pp}_l^{lk} - \sum_{\ii \in Q_k} (p_{i_1} \dots p_{i_k})^l &= \sum_{\ii \in \I^k} (p_{i_1} \dots p_{i_k})^l - \sum_{\ii \in Q_k} (p_{i_1} \dots p_{i_k})^l\\
 & = \sum_{\underset{\text{power}}{\ii \in \I^k}} (p_{i_1} \dots p_{i_k})^l\\
& \leq \sum_{d\geq 2, d|k} \; \; \sum_{u_1, \dots ,u_{k/d} \in \I} (p_{u_1} \dots p_{u_{k/d}})^{dl}\\
& \leq \sum_{d\geq 2, d|k} \norm{\pp} _{dl}^{lk}.
\end{align*}
Recall now that the function that maps $r \in [1, +\infty]$ to $\norm{\pp}_r$ is decreasing, since $\pp$ is non degenerate. Thus,
$$\norm{\pp}_l^{lk} - \sum_{\ii \in Q_k} (p_{i_1} \dots p_{i_k})^l \leq \sum_{d\geq 2, d|k} \norm{\pp}_{2l}^{lk} \leq 2 \sqrt{k} \norm{\pp}_{2l}^{lk} = o\left( \norm{\pp}_l^{lk}  \right)$$
since $\norm{\pp}_{2l} < \norm{\pp}_l$.
Finally, by dividing $\norm{\pp}_l^{kl}$, we get the stated equivalent.
\end{proof}

\subsubsection{Application to cycle count distribution}

If $(x_1, \dots, x_k)$ is a cycle of $\std(G)$ satisfying the condition $G_{x_1} = i_1, \dots, G_{x_k} = i_k$, we say that $(x_1, \dots, x_k)$ is a cycle of type $\ii = (i_1, \dots , i_k)$ of $\std(G)$. Since the cycle $(x_1, \dots, x_k)$ has $k$ different representations, corresponding to rotations of its elements, we say that the cycle is of type $\tilde{\ii}$, where $\tilde{\ii}$ is the conjugacy class of $\ii$. For any $\ii = (i_1, \dots ,i_k) \in \I^k$, define $L_{\ii} = L_{i_k}\circ L_{i_{k-1}}\circ  \dots \circ L_{i_1}$. Remark that, as for the runs $L_j,j\in \I$, we have $L_{\ii}(x) - L_{\ii}(x-1) \in \{ 0;1 \}$ for any $x \in \llbracket 1,n \rrbracket$. Consequently, $L_{\ii}$ has at least one fixed point and its fixed points are consecutive. Recall that if $(x_1  \dots  x_k)$ is a cycle of type $\ii$ of $\std(G)$, then $x_1$ and $x_1-1$ are both fixed points of $L_{\ii}$.

\begin{proposition}\label{cyclesmotsprimitifs}
Let $k \in \N^*$ and let $x_1 \in \llbracket 1 , n \rrbracket$ be such that $\std(G)^k(x_1) = x_1$. Define $x_2 = \std(G)(x_1), \dots ,x_k = \std(G)(x_{k-1})$, and $i_1 = G_{x_1},  \dots , i_k = G_{x_k}$. Then, the size $d$ of the cycle of $x_1$ in $\std(G)$ is the length of the root of the word $\ii = (i_1, \dots ,i_k)$.
\end{proposition}

In other words, if a $k$-cycle of $\std(G)$ is of type $\ii$, then $\ii$ is a primitive word. Recall that, since a $k$-cycle can be written in $k$ ways by rotating the elements of its support, we say that this cycle is of type $\tilde{\ii}$. For these reasons, to study the cycles of $\std(G)$, we will study its cycles of type $\ii$ for any $\ii \in \widetilde{Q}_k$.

\begin{proof}[Proof of Proposition \ref{cyclesmotsprimitifs}]
Write $\ii = \mathbf{r}^{k/d'}$ where $\mathbf{r}$ is the root of $\ii$, and $d'$ is the length of $\mathbf{r}$. We then have $L_{\ii} = L_{\mathbf{r}}^{k/d'}$. As for any $x, L_{\mathbf{r}}(x) - L_{\mathbf{r}}(x-1) \in \{0;1\}$, the function $L_{\mathbf{r}} (x) -x$ is positive, then zero, then negative. It follows that the sequence $x_1, L_{\mathbf{r}}(x_1), \dots , L_{\mathbf{r}}^{k/d'}(x_1)$ is monotone. Now, $L_{\mathbf{r}}^{k/d'}(x_1) = L_{\ii}(x_1) = x_1$, hence the sequence is constant. We then have in fact $L_{\mathbf{r}}(x_1) = x_1$ which means that the size $d$ of the orbit of $x_1$ divides $d'$.\\
Suppose now that $d < d'$. Since $x_{d+1}=x_1$, we would get that $\ii$ is a power of $ i_1\dots i_d$, and so is $\mathbf{r}$, which is a contradiction with the fact that $\mathbf{r}$ is the root of $\ii$. Hence, $d = d'$.
\end{proof}

\begin{proposition}
Let $\ii \in Q_k$, and $x_1\in \llbracket 1,n \rrbracket, x_2 = L_{i_1}(x_1), \dots ,x_k = L_{i_{k-1}}(x_{k-1})$. Then, $(x_1, \dots , x_k)$ is a $k$-cycle of $\std(G)$ of type $\ii$ if and only if $x_1$ and $x_1 -1$ are both fixed points of $L_{\ii}$.
\end{proposition}

\begin{proof}
If $(x_1, \dots , x_k)$ is a $k$-cycle of $\std(G)$ of type $\ii$, we have already seen that $x_1$ is a fixed point of $L_{\ii}$. Moreover, since $G_{x_1}=i_1$, $L_{i_1}(x_1-1)= L_{i_1}(x_1)-1 = x_2-1$. We repeat for $i_2, \dots, i_k$ and we get $L_{\ii}(x_1-1) = x_1-1$.\\
Conversely, if $x_1$ and $x_1-1$ are both fixed points of $L_{\ii}$, since the $L_{i_j}(x)$ rise by 0 or 1 at each step, we necessarily get that for any $j \geq 1, L_{i_j}(x_j) = x_{j+1}$ and $L_{i_j}(x_j-1) = x_{j+1} -1$, which implies that $\std(G)(x_j) = x_{j+1}$ for any $j$. Hence, $(x_1, \dots , x_k)$ is a cycle of $\std(G)$.
\end{proof}

Finally, the following result shows that any primitive word with length $k$ can lead to a $k$-cycle.

\begin{proposition}\label{existenceuniquecycle}
Let $k \geq 1$ and $\ii \in Q_k$. There exists a unique $g \in \I^k$ such that $\std(g)$ is a cycle $(x_1 , \dots , x_k)$ of type $\ii$.
\end{proposition}

\begin{proof}
Define $\tilde{\ii} = i_ki_{k-1} \dots i_1$ to be the reverse word of $\ii$ and, for any $j \in \llbracket 1 , k \rrbracket$, $$ \tilde{\ii}^{(j)} = i_{j-1}\dots i_1 i_k \dots i_j $$
the conjugate of $\tilde{\ii}$ ending with $i_j$. In what follows, the conjugacy class of $\tilde{\ii}$ is endowed with the lexicographical order $\prec$. From Proposition \ref{reverseword}, we know that $\tilde{\ii}$ is primitive, and hence its conjugates are pairwise distinct (Proposition \ref{conjuguesPrim}), then we can sort them in increasing order. Let $x_j$ be the rank of $\tilde{\ii}^{(j)}$ in this sorted sequence. 

\textit{Existence.} Define $g \in \I^k$ by $g_{x_j} = i_j$ for any $j \in \llbracket 1 , k \rrbracket$. We claim that $\sigma = \std(g)$ is the cycle $(x_1 , \dots , x_k)$. Recall that, for $j \in \llbracket 1 , k \rrbracket$, $\sigma(x_j)$ is the rank of $(i_j, x_j)$ in the sequence $\left( (i_m, x_m) \right)_m$ sorted in increasing order according to the lexicographical order $\prec$. For $j \neq m$, 
$$
(i_j, x_j) \prec (i_m, x_m) \iff (i_j, \tilde{\ii}^{(j)}) \prec (i_m, \tilde{\ii}^{(m)}) \iff i_j < i_m \text{ or } \left( i_j = i_m \text{ and } \tilde{\ii}^{(j)} \prec  \tilde{\ii} ^{(m)} \right).
$$
Now, remark that the last assertion is exactly the definition of $ \tilde{\ii} ^{(j+1)} \prec  \tilde{\ii} ^{(m+1)} $, which is equivalent to $x_{j+1}<x_{m+1}$ (with the convention $\tilde{\ii}^{(k+1)} = \tilde{\ii}^{(1)}$ and $x_{k+1} = x_1$). Hence, $$ (i_j, x_j) \prec (i_m, x_m) \iff x_{j+1} < x_{m+1} . $$
Then, for any $j \in \llbracket 1 , k \rrbracket$, the rank of $(i_j, x_j)$ in the sequence $\left( (i_m, x_m) \right)_m$ is $x_{j+1}$, and it follows that $\sigma(x_j) = x_{j+1}.$

\textit{Uniqueness.} Let $g' \in \I^k$ be a sequence and let $(y_1 , \dots , y_k)$ be a $k$-cycle such that $\std(g')$ is the cycle $(y_1 , \dots , y_k)$, and such that for any $j \in \llbracket 1 , k \rrbracket$, $g'_{y_j} = i_j$. We will show that for any $j \in \llbracket 1 , k \rrbracket , x_j = y_j$, and the equality $g = g'$ will then follow from the relation $g'_{y_j} = i_j = g_{x_j}$. Since $\std(g')(y_j) = y_{j+1}$, the number $y_{j+1}$ is the rank of $(i_j, y_j)$ in the sequence $\left( (i_m, y_m) \right)_m$ sorted in increasing order according to the lexicographical order $\prec$. Then, for $j \neq m$, 
\begin{multline*}
y_{j+1} < y_{m+1} \iff (i_j,  t_j) \prec (i_m, y_m) \iff i_j < i_m \text{ or } \left( i_j = i_m \text{ and } y_j < y_m \right).
\end{multline*}
For the second case, we repeat the process with $y_j < y_m$: 
$$ y_j < y_m \iff i_{j-1} < i_{m-1}  \text{ or } \left( i_{j-1} = i_{m-1} \text{ and } y_{j-1} < y_{m-1} \right) .$$
Since $\tilde{\ii}$ is primitive, this process ends, when we treat the case of two different $i_s$ values. Finally, we have the equivalence $$ y_{j+1} < y_{m+1} \iff \tilde{\ii}^{(j+1)} \prec \tilde{\ii}^{(m+1)} \iff x_{j+1} < x_{m+1} . $$
Hence, $x_j = y_j$ for any $j \in \llbracket 1,k \rrbracket$.
\end{proof}

\begin{example}\label{exemplecycle}Let $\ii = 427254$. We want to find $g$ such that $\std(g)$ is a cycle of type $\ii$. The conjugates of the reverse word of $\ii$ are 
\begin{align*}
\tilde{\ii}^{(1)} = 452724; \quad \tilde{\ii}^{(2)} = 445272; \quad  \tilde{\ii}^{(3)} = 244527;\\
\tilde{\ii}^{(4)} = 724452; \quad \tilde{\ii}^{(5)} = 272445; \quad \tilde{\ii}^{(6)} = 527244,
\end{align*}
ordered lexicographically as $ \tilde{\ii}^{(3)} \prec \tilde{\ii}^{(5)} \prec \tilde{\ii}^{(2)} \prec  \tilde{\ii}^{(1)} \prec \tilde{\ii}^{(6)} \prec \tilde{\ii}^{(4)}$. We then have the ranks $x_1 = 4; x_2 = 3; x_3 = 1; x_4 = 6; x_5 = 2, x_6 = 5$. We define the sequence $g$ by $g_{x_j} = i_j$, i.e.~$g_4 = 4; g_3 = 2; g_1=7; g_6=2; g_2=5; g_5=4$, we then get the sequence $$g = (7,5,2,4,4,2).$$
The standardized permutation is $$\std(g) = 651342,$$which is indeed the cycle $(x_1, x_2, x_3, x_4, x_5, x_6) = (4,3,1,6,2,5)$, of type $\ii$ in $\std(g)$.
\end{example}

\begin{definition}
Let $X$ and $Y$ be two finite sets of real numbers of the same cardinality. Let $f : X \to Y$ be the unique increasing function from $X$ to $Y$. Let $\sigma$ be a permutation of $X$ and $\pi$ be a permutation of $Y$. We say that the permutations $\sigma$ and $\pi$ are isomorphic if $\sigma = f^{-1} \circ \pi \circ f$.
\end{definition}

For $\ii \in Q_k$, define $$\mathcal{D}_{\ii} = \left\{ x \in \llbracket 1,n \rrbracket \mid L_{\ii} (x) = x, L_{\ii}(x-1) = x-1   \right\}.$$We then get that $|\mathcal{D}_{\ii}|$ is the number of $k$-cycles $(x_1 , \dots , x_k)$ of $\std(G)$ of type $\ii$. Since replacing $\ii$ with one of its conjugates does not change the corresponding cycles, we write $D_{\ii} =   |\mathcal{D}_{\ii}|$ for $\ii \in \widetilde{Q}_k$, which is well defined. Note that for any $k \geq 1$, $$ c_k = \sum_{\ii \in \tilde{Q}_k} D_{\ii} = \frac{1}{k} \sum_{\ii \in Q_k} D_{\ii} $$where $c_k$ denotes the number of cycles of $\std(G)$ of length $k$. We now define a map which aims to "add" a cycle of type $\ii$ in $\std(g)$. For any positive integers $n$ and $k$, and for any $\ii \in Q_k$, define the map $$ \psi_{\ii}^n : \I^n \to \I^{n+k}$$ as follows: for $g \in \I^n$, let $a$ be the smallest fixed point of $L_{\ii}^g$, define $a_1 = a$ and for any $j \in \llbracket 1,k-1 \rrbracket$, define $a_{j+1} = L_{i_j}^g(a_j)$. Note that $L_{i_k}^g(a_k) = a = a_1$. We will insert in $g$ the letters of $\ii$ after the $a_1$th, $a_2$th, \dots, and $a_k$th elements of $g$. Note that some $a_j$'s can be equal, in that case, if the value $b$ appears exactly $t$ times among the $a_j$'s, we will insert $t$ elements after the $b$th element of $g$. Now that we defined the insertion positions in $g$, the insertion order of the $i_j$ is given by their order in the unique $\tilde{g} \in \I^k$ constructed in the Proposition \ref{existenceuniquecycle} (see Example \ref{ajoutcycle}). The resulting sequence is then denoted $\psi_{\ii}^n(g)$.

\begin{example}\label{ajoutcycle}We revisit the example $g = (6,1,5,3,3,1,2)$ from the introduction. We want to introduce a cycle of type $\ii = 427254$ in $\std(g)$. We first define the insertion positions of the letters of $\ii$ in $g$. Since $L^g_4$ is constant equal to 5 (see Figure \ref{Nouvelleperm}), $L_{\ii}^g$ is constant equal to 5, and then the smallest (actually, the unique) fixed point of $L_{\ii}^g$ is 5. Then we will insert the letters of $\ii$ after the positions 5, $L_4^g(5) = 5, L_2^g(5) = 2, L_7^g(2) = 7, L_2^g(7) = 3$ and $L_5^g(3) = 6$. The sequence $\psi_{\ii}^7(g)$ is then of the form
$$ \psi_{\ii}^7(g) = (6, 1, \hphantom{7} , 5, \hphantom{5}, 3, 3, \hphantom{2}, \hphantom{4}, 1, \hphantom{4},  2, \hphantom{2}). $$
The order of insertion is that of the unique $\tilde{g}$ associated with $\ii$ in Proposition \ref{existenceuniquecycle}. We have seen in Example \ref{exemplecycle} that $\tilde{g} = (7,5,2,4,4,2)$, and finally
$$ \psi_{\ii}^7(g) = (6, 1, {\color{red}7}, 5, {\color{red}5}, 3, 3, {\color{red}2}, {\color{red}4}, 1, {\color{red}4}, 2, {\color{red}2}). $$
We obtain the permutation $$ \std(\psi_{\ii}^7(g)) = 12 ~ 1 ~ {\color{red}13} ~ 10 ~ {\color{red}11}~6~7~{\color{red}3}~{\color{red}8}~2~{\color{red}9}~4~{\color{red}5}.$$
As we can see in more generality below, the inserted values create a cycle of type $\ii$, which is $(3, 13, 5, 11, 9, 8)$, and the rest of the permutation is a permutation of $\{ 1,2,4,6,7,10,12\}$, such that $\restriction{\std(\psi_{\ii}^7(g))}{\{ 1,2,4,6,7,10,12\}}$ is isomorphic to $\std(g)$ (see Figure \ref{Nouvelleperm}). Now note that we have $\std(g) = 7164523 = (1, 7, 3, 6, 2)(4)(5)$, and thus we have $D^g_{62511} = 1$, $D^g_3 = 2$ and $D_{\ii}^g = 0$. For the new permutation, we have $\std(\psi_{\ii}^7(g)) = (1, 12, 4, 10, 2)(6)(7){\color{red}(3, 13, 5, 11, 9, 8)}$, and thus $D^{\psi_{\ii}^7(g)}_{62511} = 1$, $D^{\psi_{\ii}^7(g)}_3 = 2$ and $D^{\psi_{\ii}^7(g)}_{\ii} = 1$.

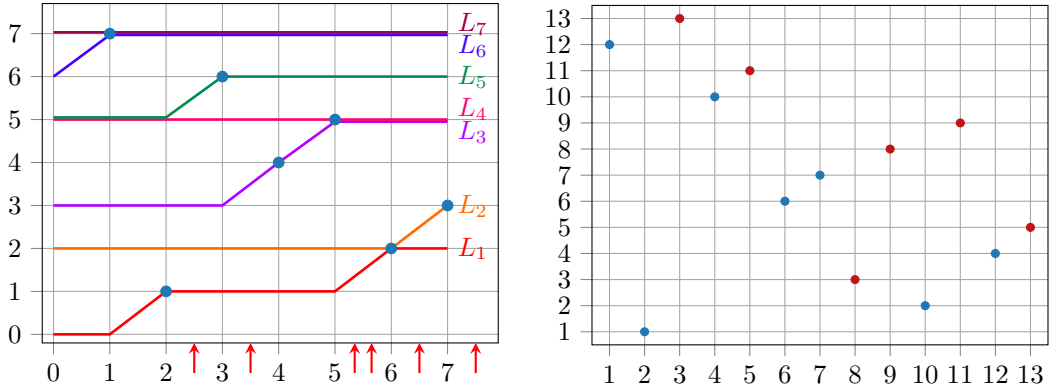
\begin{figure}[!h]
\centering
\begin{minipage}[c]{.48\linewidth}
        \centering
       \begin{tikzpicture}
    \begin{axis}[
        scale only axis,  
        width=6cm,        %
        height=4.5cm,       %
        xmin=-0.2, xmax=7.9,
        ymin=-0.2, ymax=7.7,
        xtick={0,1,2,3,4,5,6,7},
        ytick={0,1,2,3,4,5,6,7},
        grid=both,
        major grid style={gray!70},
        axis on top=false,
        tick align=outside,
        tick pos=left,
        clip=false
    ]

    \addplot[
        color=run1,      %
        line width=1pt,  %
        no markers     %
    ] coordinates {
    	(0, 0)
    	(1, 0)
        (2, 1)
        (5, 1)
        (6, 2)
        (7, 2)
    }
    node[right, pos=1] {$L_1$};
    \addplot[
        only marks,       
        mark=*,           %
        mark size=2pt,  %
        color=mplblue     %
    ] coordinates {
        (2, 1)
        (6, 2)
    };
    \addplot[
        color=run2,      %
        line width=1pt,  %
        no markers     %
    ] coordinates {
    	(0, 2)
    	(6, 2)
    	(7, 3)
    }
    node[right, pos=1] {$L_2$};
	\addplot[
        only marks,       
        mark=*,           %
        mark size=2pt,  %
        color=mplblue     %
    ] coordinates {
        (7, 3)
    };
    \addplot[
        color=run3,      %
        line width=1pt,  %
        no markers     %
    ] coordinates {
    	(0, 3)
    	(3, 3)
    	(4, 4)
        (5, 4.95)
        (7, 4.95)
    }
    node[pos=1,right,yshift=-4pt] {$L_3$};
	\addplot[
        only marks,       
        mark=*,           %
        mark size=2pt,  %
        color=mplblue     %
    ] coordinates {
        (4, 4)
        (5, 5)
    };
    \addplot[
        color=run4,      %
        line width=1pt,  %
        no markers     %
    ] coordinates {
    	(0, 5)
        (7, 5)
    }
    node[pos=1,right,yshift=4pt] {$L_4$};
    \addplot[
        color=run5,      %
        line width=1pt,  %
        no markers     %
    ] coordinates {
    	(0, 5.05)
    	(2, 5.05)
    	(3, 6)
        (7, 6)
    }
    node[pos=1,right,yshift=0pt] {$L_5$};
	\addplot[
        only marks,       
        mark=*,           %
        mark size=2pt,  %
        color=mplblue     %
    ] coordinates {
        (3, 6)
    };
    \addplot[
        color=run6,      %
        line width=1pt,  %
        no markers     %
    ] coordinates {
    	(0, 6)
    	(1, 6.97)
    	(7, 6.97)
    }
    node[pos=1,right,yshift=-4pt] {$L_6$};
    \addplot[
        color=run7,      %
        line width=1pt,  %
        no markers     %
    ] coordinates {
    	(0, 7.03)
    	(7, 7.03)
    }
    node[pos=1,right,yshift=3pt] {$L_7$};
	\addplot[
        only marks,       
        mark=*,           %
        mark size=2pt,  %
        color=mplblue     %
    ] coordinates {
        (1, 7)
    };
    \draw[red, -stealth, thick] (axis cs:2.5, -0.9) -- (axis cs:2.5, -0.2);
    \draw[red, -stealth, thick] (axis cs:3.5, -0.9) -- (axis cs:3.5, -0.2);
	\draw[red, -stealth, thick] (axis cs:5.35, -0.9) -- (axis cs:5.35, -0.2);
	\draw[red, -stealth, thick] (axis cs:5.65, -0.9) -- (axis cs:5.65, -0.2);
	\draw[red, -stealth, thick] (axis cs:6.5, -0.9) -- (axis cs:6.5, -0.2);
	\draw[red, -stealth, thick] (axis cs:7.5, -0.9) -- (axis cs:7.5, -0.2);

    \end{axis}
	\end{tikzpicture}
    \end{minipage}
    \begin{minipage}[c]{.48\linewidth}
        \centering
        \begin{tikzpicture}
    \begin{axis}[
        scale only axis,  
        width=6cm,        %
        height=4.5cm,       %
        xmin=0.5, xmax=13.5,
        ymin=0.5, ymax=13.5,
        xtick={1,2,3,4,5,6,7,8,9,10,11,12,13},
        ytick={1,2,3,4,5,6,7,8,9,10,11,12,13},
        grid=both,
        major grid style={gray!70},
        axis on top=false,
        tick align=outside,
        tick pos=left
    ]

    \addplot[
        only marks,       
        mark=*,           %
        mark size=1.5pt,  %
        color=mplblue     %
    ] coordinates{
    (1,12)
    (2,1)
    (4,10)
    (6,6)
    (7,7)
    (10,2)
    (12,4)
    };
	\addplot[
        only marks,       
        mark=*,           %
        mark size=1.5pt,  %
        color=newred     %
    ] coordinates{
    (3,13)
    (5,11)
    (8,3)
    (9,8)
    (11,9)
    (13,5)
    };

    \end{axis}
	\end{tikzpicture}
    \end{minipage}
\caption{The permutations $\std(g)$ and $\std(\psi_{\ii}^7(g))$ for $g = (6,1,5,3,3,1,2)$ and $\ii = 427254$. The red arrows represent the positions of insertion. The red points represent the inserted cycle of type $\ii$. The blue points form a permutation of $\{ 1,2,4,6,7,10,12 \}$ which is isomorphic to the permutation $\std(g)$ on the left.}\label{Nouvelleperm}
\end{figure}
\end{example}

The above example illustrates the following Proposition:

\begin{proposition}\label{bijectionbase}
The map $\psi_{\ii}^n$ is a bijection from $\I^n$ to $\left\{ g' \in \I^{n+k} \mid D_{\ii}(g') \geq 1 \right\}$ such that for any $g \in \I^n$, $D_{\ii}(\psi_{\ii}^n(g)) = D_{\ii}(g)+1$ and for any primitive word $\mathbf{j}$ not conjugate to the word $\ii$, $D_{\mathbf{j}}(\psi_{\ii}^n(g)) = D_{\mathbf{j}}(g)$. More precisely, if $x_j$ is the position of $i_j$ in $\psi_{\ii}^n(g)$, then $(x_1, \dots, x_k)$ is a cycle of type $\ii$ of $\std \left( \psi_{\ii}^n(g) \right)$, and $\restriction{\std \left( \psi_{\ii}^n(g) \right)}{\left\{ x_1, \dots, x_k \right\} }$ is isomorphic to $\std(g)$.
\end{proposition}

The above Proposition means that the map $\psi_{\ii}^n$ adds a cycle of type $\ii$ in the permutation $\std(G)$. Note that, by construction, if $G_1, \dots, G_{n+k}$ are i.i.d.~random variables taking values in $\I$, then for any $g \in \I^n$ and $\ii \in Q_k$, $\Proba(G = \psi_{\ii}^n(g)) = p_{\ii} \Proba\left( \restriction{G}{\llbracket 1,n \rrbracket} = g \right) $.

Some lemmas will be useful to prove Proposition \ref{bijectionbase}.

\begin{lemma}\label{ptfixerecurrence}
Let $g \in \I^n$, let $\ii \in Q_k$. Let $a_1$ be the smallest fixed point of $L_{\ii}^g$, and for any $j \in \llbracket 1,k-1 \rrbracket$, define $a_{j+1} = L_{i_j}^g(a_j)$. We denote by $\ii^{(j)}$ the conjugate of $\ii$ starting with $i_j$. Then, $a_j$ is the smallest fixed point of $L_{\ii^{(j)}}$.
\end{lemma}

\begin{proof}
Let $j \geq 2$. We have $$ L_{\ii^{(j)}}(a_j) = L_{\ii^{(j)}} \circ L_{i_{j-1}} \circ \dots \circ L_{i_1} (a_1) = L_{i_{j-1}} \circ \dots \circ L_{i_1} \circ L_{\ii} (a_1) = L_{i_{j-1}} \circ \dots \circ L_{i_1} (a_1) = a_j $$
and hence, $a_j$ is a fixed point of $L_{\ii^{(j)}}$. Now, if $b$ is a smaller fixed point of $L_{\ii^{(j)}}$, then $$ L_{\ii} \circ L_{i_k} \circ \dots \circ L_{i_j}(b) = L_{i_k} \circ \dots \circ L_{i_j}(b) $$ and hence $c = L_{i_k} \circ \dots \circ L_{i_j}(b)$ is a fixed point of $L_{\ii}$. We have $$L_{i_{j-1}} \circ \dots \circ L_{i_1}(c) = b < a_j = L_{i_{j-1}} \circ \dots \circ L_{i_1}(a_1)$$ and then, since $L_{i_{j-1}} \circ \dots \circ L_{i_1}$ is non-decreasing, $c<a_1$ which is a contradiction with the definition of $a_1$.
\end{proof}

\begin{lemma}\label{ordreptsfixes}
\begin{itemize}
\item Let $g \in \I^n$. For any $\ii, \jj \in \I^k$, if $\ii \neq \jj$ and $\tilde{\ii} \prec \tilde{\jj}$ (where $\tilde{\ii}$ and $\tilde{\jj}$ are the reverse words of $\ii$ and $\jj$), then for any $x,y \in \llbracket 1,n \rrbracket$, $L_{\ii}(x) \leq L_{\jj}(y)$.
\item Let $g \in \I^n$, let $\ii \in Q_k$. We denote by $\ii^{(j)}$ the conjugate of $\ii$ starting with $i_j$, and by $\tilde{\ii}^{(j)}$ the reverse word of $\ii^{(j)}$. We also denote by $a_j$ the smallest fixed point of $L_{\ii^{(j)}}$. Then, if $j \neq l$ and $\tilde{\ii}^{(j)} \prec \tilde{\ii}^{(l)}$, it holds $a_j \leq a_l$.
\end{itemize}
\end{lemma}

\begin{proof}
\begin{itemize}
\item Let $l \in \llbracket 1,k \rrbracket$ be such that $i_l < j_l$ and $i_m = j_m$ for any $m > l$. Let $z,t \in \llbracket 1,n \rrbracket$. By the definition of runs, it holds $L_{i_l}(z) \leq L_{j_l}(0) \leq L_{j_l}(t)$. Then, by applying the non-decreasing map $L_{i_k} \circ \dots \circ L_{i_{l+1}} = L_{j_k} \circ \dots \circ L_{j_{l+1}}$, $$L_{i_k} \circ \dots \circ L_{i_l}(z) \leq L_{j_k} \circ \dots \circ L_{j_l}(t). $$ We get the result by replacing $z$ by $L_{i_{l-1}} \circ \dots \circ L_{i_1}(x)$ and $t$ by $L_{j_{l-1}} \circ \dots \circ L_{j_1}(y)$.
\item From the first point, it follows that if $\tilde{\ii}^{(j)} \prec \tilde{\ii}^{(l)}$ (where $j \neq l$), then $$ a_j = L_{\ii^{(j)}}(a_j) \leq L_{\ii^{(l)}}(a_l) = a_l,$$hence the result holds.
\end{itemize}

\end{proof}

\begin{proof}[Proof of Proposition \ref{bijectionbase}]
Let $g \in \I^n$. We first prove that the map $\psi_{\ii}^n$ indeed inserts a $k$-cycle. Denote by $x_j$ the position of the inserted letter $i_j$ in $\psi_{\ii}^n(g)$. We use the conventions $i_{k+1} = i_1$, $x_{k+1} = x_1$ and $a_{k+1} = a_1$. We write $g' = \psi_{\ii}^n(g)$ and $\sigma = \std(g')$. We prove that $(x_1, \dots, x_k)$ is a cycle of $\sigma$. Let $j \in \llbracket 1,k \rrbracket$, we prove that $\sigma(x_j) = x_{j+1}$. Denote by $P$ the set $\{x_1, \dots, x_k \}$. Denote by $(y_1, \dots, y_k)$ the cycle $\std(\tilde{g})$, and define $y_{k+1} = y_1$. Note that 
$$ \sigma(x_j) = \# \left\{ s \mid g'_s < g'_{x_j} \right\} + \# \left\{ s \leq x_j \mid g'_s = g'_{x_j} \right\} . $$
We split this expression in two parts: the indices of elements coming from $g$ (i.e.~the indices of the complementary $P^c$ of $P$), and the indices of the inserted elements (i.e.~the indices of $P$):
\begin{multline*} 
\sigma(x_j) = \underbrace{\left( \# \left\{ s \in P^c \mid g'_s < g'_{x_j} \right\} + \# \left\{ s \in P^c \mid s \leq x_j \text{ and } g'_s = g'_{x_j} \right\} \right)}_{\text{contribution of }g} \\+ \underbrace{\left( \# \left\{ s \in P \mid g'_s < g'_{x_j} \right\} + \# \left\{ s \in P \mid s \leq x_j \text{ and } g'_s = g'_{x_j} \right\} \right)}_{\text{contribution of }\tilde{g}}. 
\end{multline*}
From Lemmas \ref{ptfixerecurrence} and \ref{ordreptsfixes}, in follows that the first term ("contribution of $g$") is exactly $L_{i_j}^g(a_j) = a_{j+1}$, which is the number of elements of $g$ before $i_{j+1}$. Then, the second term ("contribution of $\tilde{g}$") is exactly $\std(\tilde{g})(y_j) = y_{j+1}$, which is the number of inserted values of $g'$ before $i_{j+1}$ (including $i_{j+1}$). Now note that it follows from Lemmas \ref{ptfixerecurrence} and \ref{ordreptsfixes} and from the construction of the $x_j$ in Proposition \ref{existenceuniquecycle} that if $x_j < x_l$ then $a_j \leq a_l$. Thus, for any $j$, $\psi^n_{\ii}$ inserts the element $i_{j}$ in $g$ between $g_{a_j}$ and $g_{a_j+1}$. Finally, $\sigma(x_j)$ is the total number of elements before $i_{j+1}$ (including $i_{j+1}$) in $g'$, which gives $\sigma(x_j) = x_{j+1}$. Thus, $(x_1, \dots, x_k)$ is a cycle of $\sigma$.

Hence, $P$ is invariant under $\sigma$. Then, $P^c$ is invariant under $\sigma$ and $\restriction{\std(\psi_{\ii}^n(g))}{P^c}$ is isomorphic to $\std(\restriction{\psi_{\ii}^n(g)}{P^c}) = \std(g)$. From this, it follows the claimed relations between 
$D_{\mathbf{j}}(g)$ and $D_{\mathbf{j}}(\psi_{\ii}^n(g))$.

The map $\psi_{\ii}^n$ is a bijection: to recover $g$ from $g' = \psi_{\ii}^n(g)$, if $(x_1, \dots, x_k)$ is the cycle of $\std(g')$ given by the two smallest fixed points of $L_{\ii}^{g'}$, it suffices to define $g$ by removing from $g'$ the elements in positions $x_1, \dots, x_k$.
\end{proof}

We can finally turn to the proof of Theorem \ref{loiDiCycles}. We write $p_{\ii} = p_{i_1}\dots p_{i_k}$.

\begin{proof}[Proof of Theorem \ref{loiDiCycles}]
We here generalize the proof of Proposition \ref{loiDi}. Fix $k_1, \dots ,k_r \geq 1$, $\ii_1 \in Q_{k_1}, \dots , \ii_r \in Q_{k_r} $ pairwise non-conjugate, and $l_1, \dots ,l_r \geq 0$ such that $\sum k_jl_j \leq n$. Write 
$$ A = \left\{ g \in \I^n \mid  D_{\ii_1}(g) \geq l_1, \dots , D_{\ii_r}(g) \geq l_r \right\} . $$
Define the map $$\psi : \I^{n - \sum k_jl_j} \to A$$ 
by $$ \psi = \psi_{\ii_r}^{n - k_r} \circ \dots \circ \psi_{\ii_1}^{n-\sum k_jl_j + (l_1-1)k_1} \circ \dots \circ \psi_{\ii_1}^{n-\sum k_jl_j} .  $$
In other words, $\psi$ is the map which adds $l_1$ cycles of type $\ii_1$, \dots, and $l_r$ cycles of type $\ii_r$ in $\std(g)$, by iterating the map of Proposition \ref{bijectionbase}. Define the map 
$$\phi : A \to \I^{n - \sum k_jl_j}$$
as the inverse of $\psi$. Now, since $G$ is a sequence of i.i.d.~variables, $$\Proba\left( D_{\ii_1} \geq l_1, \dots , D_{\ii_r} \geq l_r  \right) = \Proba(G \in A) =  p_{\ii_1}^{l_1}\dots p_{\ii_r}^{l_r}\Proba \left( \restriction{G}{\llbracket 1, n - \sum k_jl_j \rrbracket} \in \I^{n - \sum k_jl_j} \right) = p_{\ii_1}^{l_1}\dots p_{\ii_r}^{l_r}$$ which is the stated result.
\end{proof}

From the case $r=1$ of Theorem \ref{loiDiCycles}, it follows an explicit formula for the expectation of $c_k$:

\begin{corollary}
For $\ii \in Q_k, D_{\ii}$ follows the distribution $\min(Geo_0(1-p_{\ii}), \left\lfloor \frac{n}{k}\right\rfloor )$. We also have $$\displaystyle \E\left(c_k\right) = \sum_{\ii\in \widetilde{Q}_k} \sum_{l=1}^{\lfloor n/k \rfloor} p_{\ii}^l =\frac{1}{k} \sum_{\ii\in Q_k} \sum_{l=1}^{\lfloor n/k \rfloor} p_{\ii}^l .$$
\end{corollary}

\section{Asymptotic distribution of small cycles}\label{sectionpetitscycles}

\subsection{Fixed-distribution case}

In this section, we consider the case where the distribution $\pp$ of the $G_k$ does not depend on $n$, and we let $n$ tend to infinity to study the asymptotic distribution of the cycles of our model. From Theorem \ref{loiDiCycles}, it immediately follows:

\begin{proposition}
\begin{itemize}
\item (Fixed points) For $i_1, \dots ,i_r \in \I$ pairwise distinct, for $ n \rightarrow +\infty$, we have
$$(D_{i_1}, \dots ,D_{i_r}) \overset{(d)}{\longrightarrow} \Geo_0(1-p_{i_1}) \otimes  \dots  \ \otimes \Geo_0(1-p_{i_r}) \, .$$
\item (Cycles) More generally, for $k_1, \dots ,k_r \geq 1, \ii_1 \in \widetilde{Q}_{k_1}, \dots , \ii_r \in \widetilde{Q}_{k_r} $ pairwise distinct, we have for $ n \rightarrow +\infty$,
$$ \left( D_{\ii_1}, \dots , D_{\ii_r} \right)  \overset{(d)}{\longrightarrow} \Geo_0(1-p_{\ii_1}) \otimes  \dots  \ \otimes \Geo_0(1-p_{\ii_r}). $$
\end{itemize}
\end{proposition}

A natural question is to ask how does the distribution of total number of cycles of a given length behave, without restricting to a finite number of primitive words. We are in particular interested in having a convergence result in $\ell^1$ since this will imply the convergence of the cycle counts $c_k$ (recall that each $c_k$ writes as an infinite sum of $D_{\ii}$'s). To do this, we need a criterion for convergence in distribution in $\ell^1(\bigcup_{l \leq k}\widetilde{Q}_l)$. This result may be well-known but we have not found a reference in the literature. In what follows, let $C$ be a countable set ($C$ will equal $\bigcup_{l \leq k}\widetilde{Q}_l$ later).

\begin{theorem}\label{ThCVloi}
Let $(X^n)_{n\in \N}$ and $X$ be random variables taking values in $\ell^1(C)$. Assume that\\
$\bullet$ the finite-dimensional distributions of $(X^n)_{n\in \N}$ converge to those of $X$\\
$\bullet$ there exists $u \in \ell^1(C)$ such that for any $n$ and $i$, $\E\left(\left|X^n_i\right|\right) \leq u_i$.\\
Then, $X^n$ converges in distribution to $X$ in $\ell^1(C)$.
\end{theorem}

Since $C$ is countable, we can prove this theorem in the particular case $C = \N$. In order to prove this theorem, we use the notion of tight sequences.

\begin{definition}
A sequence $(X^n)_{n\in \N}$ of random variable taking values in a complete metric space $E$ is tight if for any $\epsilon > 0$, there exists a compact subset $K$ of $E$ such that for any $n \in \N, \Proba \left( X_n \in K \right) \geq 1- \epsilon$. 
\end{definition}

In that case, we have this useful result, following from Prokhorov theorem:

\begin{theorem}\cite[Theorem 5.1]{BillCVprobMeas}\label{prokhorov}
Let $(X^n)_{n\in \N}$ be a sequence of random variables taking values in a complete separable metric space $E$, and $Z$ a random variable taking values in $E$. If $X_n$ is tight, and if any subsequence of $X_n$ converging in distribution converges to $Z$, then $X_n$ converges in distribution to $Z$.
\end{theorem}

As we would like to apply the previous theorem to $E = \ell^1(\N)$, we need a compactness criterion on $\ell^1(\N)$.

\begin{proposition}\label{compactsl1}
A subset $K$ of $\ell^1(\N)$ is compact if and only if $K$ is closed, bounded and for any $\delta > 0$, there exists $N \geq 0$ such that for any $\displaystyle u \in K, \sum_{n \geq N}\left|u_n\right| < \delta$.
\end{proposition}

\begin{proof}
Suppose that $K$ is a compact subset of $\ell^1(\N)$. Clearly, $K$ is closed and bounded. For $N \geq 0$ and $u \in K$, we write $$F_N(u) = \sum_{n \geq N}|u_n|.$$
The functions $F_N$ are continuous (and even 1-Lipschitz) on $K$: we indeed have $$ \left| F_N(u) - F_N(v) \right| = \left| \sum_{n \geq N} (|u_n|-|v_n|) \right| \leq \sum_{n \geq N} |u_n - v_n| \leq \norm{u-v}_1.$$ 
Moreover, the function sequence $(F_N)_{N \geq 0}$ is non-increasing (i.e.~for any $u \in K$ and $N \geq 0$, $F_N(u) \geq F_{N+1}(u)$), and converges pointwise to 0, hence Dini's theorem gives us that the convergence is uniform on the compact $K$. We thus have the direct implication.

Conversely, assume that $K$ is closed, included in the ball $B(0,R)$, and that for any $\delta >0$, there exists $N \geq 0$ such that for any $\displaystyle u \in K, \sum_{n \geq N}\left|u_n\right| < \delta$. We fix $\delta > 0$ and such an $N$. We then have that $K$ is included in $H_N + B(0,\delta)$, where $H_N = \left\{ u \in  \ell^1(\N)\mid  \forall n \geq N, u_n= 0 \right\}$ which is finite-dimensional. The ball $\overline{B_{H_N}(0,\delta + R)}$ is a compact subset of $H_N$, thus there exists $M \geq 1$ and $y^1,...,y^M \in H_N$ such that $$\overline{B_{H_N}(0,\delta + R)} \subset \bigcup_{j=1}^M B_{H_N}(y^i,\delta).$$
Let $x \in K$. There exists $y \in H_N$ such that $\norm{x-y}_1 < \delta$. Then, we get that $\norm{y}_1-\norm{x}_1 < \delta$, thus $\norm{y}_1 < \delta + \norm{x}_1 < \delta + R $. Therefore, $y \in B_{H_N}(0,\delta+R)$, then there exists $j \in {1,...,M}$ such that $\norm{y - y^j}_1<\delta$. Then it holds $$ \norm{x-y^j}_1 \leq \norm{x-y}_1 + \norm{y-y^j}_1 < 2\delta.$$ Finally, 
$$K \subset \bigcup_{j=1}^M B(y^j,2\delta).$$ 
We have proved that the set $K$ is totally bounded in the complete space $\ell^1(\N)$, hence it is compact (see e.g.~\cite[19.18]{zbMATH01059780}).
\end{proof}

We now turn to the proof of Theorem \ref{ThCVloi}.

\begin{proof}[Proof of Theorem \ref{ThCVloi}]
Since $C$ is countable, we assume that $C = \N$. Let $\epsilon > 0$. On the one hand, we have for $A > 0$, $$\Proba\left(\norm{X^n}_1> A\right) \leq \frac{\E\left(\norm{X^n}_1\right)}{A} \leq \frac{\sum_{i \geq 0}{\E(|X_i^n|)}}{A} \leq \frac{\norm{u}_1}{A} \leq \frac{\epsilon}{2}  $$ for $A$ chosen large enough, independently of $n$. On the other hand, for $k \geq 1$, $$ \Proba\left(\sum_{i\geq N}|X^n_i| \geq \frac{1}{k}\right) \leq k \sum_{i \geq N}|u_i|,$$ hence there exists $N_k \in \N$ independent of $n$ such that $\displaystyle \Proba\left(\sum_{i\geq N_k}|X^n_i| \geq \frac{1}{k}\right) \leq \frac{\epsilon}{2^{k+1}}$.\\
Let us set $$K = \left\{  v \in l^1(\N), \norm{v}_1 \leq A \text{ and } \forall k \geq 1, \sum_{i \geq N_k} |v_i| \leq \frac{1}{k}  \right\}$$ which is compact by Proposition \ref{compactsl1}. Then, we have $\Proba(X^n \in K) \geq 1-\epsilon$, hence the sequence $(X^n)$ is tight. The result comes from Theorem \ref{prokhorov} and from the fact that the finite-dimensional distributions of the required limit characterizes its distribution.
\end{proof}

Let us return to our model of random permutations. From the previous theorem, it follows a result on the cycles of $\std(G)$ with length smaller than any fixed number $k$ as $n$ tends to infinity (Theorem \ref{CVloiDiloifixe}). Recall that $c_k \left(\std(G) \right) = \sum_{\ii \in \widetilde{Q}_k} D_{\textbf{i}} .$ We can now prove Theorem \ref{CVloiDiloifixe}.

\begin{proof}[Proof of Theorem \ref{CVloiDiloifixe}]
First, note that $p_i \leq 1/2$ excepted for at most one $i$. Then, if $l=1$, $\frac{p_i}{1-p_i} \leq 2p_i$ excepted for at most one $i$, and if $l \geq 2$, $p_\ii$ is a product with at least two $p_j$ distinct factors with sum less than or equal to 1, hence $p_{\ii} \leq 1/2$ (even 1/4) for any $\ii \in Q_l$, then $\frac{p_{\ii}}{1-p_{\ii}} \leq 2p_{\ii}$. Hence, for any $l \leq k$, $\sum_{\ii \in Q_l} \frac{p_{\ii}}{1-p_{\ii}} < +\infty$.

We claim that $\underset{\ii \in \widetilde{Q}_l, l \leq k}{\bigotimes} \Geo_0(1-p_{\ii})$ is a distribution taking values in the set $\ell^1\left(\bigcup_{l \leq k}\widetilde{Q}_l\right)$: indeed, Fubini-Tonelli's theorem gives for i.i.d.~$X_{\ii}$ with distribution $\Geo_0(1-p_{\ii})$,
$$ \E \left(\sum_{l = 1}^k \sum_{\ii \in \widetilde{Q}_l} X_{\ii}  \right) = \sum_{l = 1}^k \sum_{\ii \in \widetilde{Q}_l} \E (X_{\ii}) = \sum_{l = 1}^k \sum_{\ii \in \widetilde{Q}_l} \frac{p_{\ii}}{1-p_{\ii}} = \sum_{l = 1}^k \frac{1}{l} \sum_{\ii \in Q_l} \frac{p_{\ii}}{1-p_{\ii}} < +\infty.$$\\
Therefore, $\sum_{i \geq 0} X_i$ is finite a.s.

Recall that it follows from Theorem \ref{loiDiCycles} that the finite-dimensional distributions of $(D_{\ii}^n)_{\ii}$ converge to those of $\underset{\ii \in \widetilde{Q}_l, l \leq k}{\bigotimes} \Geo_0(1-p_{\ii})$ as $n$ tends to infinity. To apply Theorem \ref{ThCVloi}, it is sufficient to note that for any $\ii \in \widetilde{Q}_l$  and any $n \geq 1, \E(D_{\ii}^n) \leq \E(X_{\ii})$ by Proposition \ref{loiDi}, and that $\bigcup_{l \leq k}\widetilde{Q}_l$ is countable.
\end{proof}

\subsection{The small frequency case}
In this part, we assume that the distribution $\pp$ varies with $n$, more precisely, we consider the case where $\displaystyle \norm{\pp^{(n)}}_{\infty} = \max_{i \in \I} p_i^{(n)} \underset{n \rightarrow +\infty}{\longrightarrow} 0$. In this case, the behavior of the short cycles is different from the previous case, and we recover the same result as in the uniform model. A few technical lemmas will be useful before proving Theorem \ref{CyclesLoiVariable}.

\begin{lemma}\label{momentsFctQueue}
Let $X_1, \dots ,X_p$ be random variables taking values in $\N$, and let $r_1, \dots ,r_p \geq 1$. Then,
$$ \E\left( X_1^{r_1} \dots X_p^{r_p} \right) = \sum_{l_1\geq 1} \dots \sum_{l_p \geq 1} \left( \prod_{j=1}^p \left[ l_j^{r_j}-\left(l_j-1\right)^{r_j}\right] \right) \Proba\left( X_1 \geq l_1, \dots ,X_p \geq l_p  \right)  .  $$
\end{lemma}

\begin{proof}
We have
$$\E\left( X_1^{r_1} \dots X_p^{r_p} \right) = \sum_{k_1 \geq 1} \dots \sum_{k_p \geq 1} \left(  \prod_{j=1}^p k_j^{r_j}\right) \Proba \left(X_1=k_1, \dots ,X_p=k_p\right).$$
Replace $k_j^{r_j}$ by the telescoping sum $\sum_{l_j = 1}^{k_j} \left[ l_j^{r_j} - (l_j-1)^{r_j}  \right]$, and then by expanding the product of these sums and changing the summation order, we get 
\begin{align*}
\E\left( X_1^{r_1} \dots X_p^{r_p} \right) & = \sum_{l_1\geq 1} \dots \sum_{l_p \geq 1} \sum_{k_1 \geq l_1} \dots \sum_{k_p \geq l_p}  \left( \prod_{j=1}^p  \left[ l_j^{r_j} - (l_j-1)^{r_j}  \right]\right) \Proba \left(X_1=k_1, \dots ,X_p=k_p\right)\\
& = \sum_{l_1\geq 1} \dots \sum_{l_p \geq 1} \left( \prod_{j=1}^p \left[ l_j^{r_j}-\left(l_j-1\right)^{r_j}\right] \right) \Proba\left( X_1 \geq l_1, \dots ,X_p \geq l_p  \right)
\end{align*}
which is the stated result.
\end{proof}

\begin{lemma}\label{LemmeTechnique1}
If $ \displaystyle \norm{\pp^{(n)}}_{\infty}  \underset{n \rightarrow +\infty}{\longrightarrow} 0$, then for $k \in \N^* $, $m \geq 1$,  $q_1, \dots ,q_m \geq 1$ and $\ii_1, \dots ,\ii_m$ pairwise distinct classes of primitive words of respective lengths $k_1, \dots, k_m \leq k$, we have
$$ \E \left( D_{\ii_1}^{q_1} \dots D_{\ii_m}^{q_m}  \right) \underset{n \rightarrow +\infty}{=} p_{\ii_1} \dots p_{\ii_m}\left( 1 + O\left(   \norm{\pp^{(n)}}_{\infty}   \right)  \right)  $$
where the error is uniform in $\ii_1, \dots, \ii_m$.
\end{lemma}

\begin{proof}
From Lemma \ref{momentsFctQueue} and Theorem \ref{loiDiCycles}, it follows that
\begin{align*}
\E \left( D_{\ii_1}^{q_1} \dots D_{\ii_m}^{q_m}  \right) &= \sum_{\underset{\sum k_il_i \leq n}{l_1, \dots ,l_m\geq 1}} p_{\ii_1}^{l_1} \dots p_{\ii_m}^{l_m} \left(\prod_{j=1}^m \left[ l_j^{q_j}-(l_j-1)^{q_j} \right] \right)\\
& =p_{\ii_1} \dots p_{\ii_m} \sum_{\underset{\sum k_il_i \leq n-k_1- \dots -k_m}{l_1, \dots ,l_m\geq 0}} p_{\ii_1}^{l_1} \dots p_{\ii_m}^{l_m} \left(\prod_{j=1}^m \left[ (l_j+1)^{q_j}-l_j^{q_j} \right] \right).
\end{align*}
It remains to prove that the sum behaves as $1 + O\left(   \norm{\pp^{(n)}}_{\infty}   \right)$. To do this, we bound the $p_i^{(n)}$ by $\norm{\pp^{(n)}}_{\infty}$, and the $q_j$ by $\sum q_j$.
\begin{multline*}
\sum_{\underset{\sum k_il_i \leq n-k_1- \dots -k_m}{l_1, \dots ,l_m\geq 0}} p_{\ii_1}^{l_1} \dots p_{\ii_m}^{l_m} \left(\prod_{j=1}^m \left[ (l_j+1)^{q_j}-l_j^{q_j} \right] \right) - 1\\ 
\leq \sum_{\underset{\text{not all zero}}{l_1, \dots ,l_m\geq 0}}  \left(\prod_{j=1}^m  \norm{\pp^{(n)}}_{\infty}^{l_i} \left[ (l_j+1)^{\sum q_j}-l_j^{\sum q_j} \right] \right)
\end{multline*}
$$
\leq \norm{\pp^{(n)}}_{\infty} m \left( \sum_{l\geq 0 } \norm{\pp^{(n)}}_{\infty}^l \left[ (l+1)^{\sum q_j}-l^{\sum q_j} \right]   \right) ^{m-1}\left( \sum_{l\geq 1 } \norm{\pp^{(n)}}_{\infty}^{l-1} \left[ (l+1)^{\sum q_j}-l^{\sum q_j} \right]   \right).
$$

Since $\norm{\pp^{(n)}}_{\infty} \leq \dfrac{1}{2}$ for $n$ large enough, the above series converge, hence the above upper-bound is a $O\left( \norm{\pp^{(n)}}_{\infty} \right)$.
\end{proof}

\begin{lemma}\label{norme2tendvers0}
If $ \displaystyle \norm{\pp^{(n)}}_{\infty}  \underset{n \rightarrow +\infty}{\longrightarrow} 0$, then for any $l \geq 2$, $ \norm{\pp^{(n)}}_l \underset{n \rightarrow +\infty}{\longrightarrow} 0.$
\end{lemma}

\begin{proof}
It holds that
$$ \norm{\pp^{(n)}}_l^l = \sum_{i \in \I} p_i^l \leq \norm{\pp^{(n)}}_{\infty} \sum_{i \in \I} p_i^{l-1} \leq \norm{\pp^{(n)}}_{\infty} \sum_{i \in \I} {p_i} = \norm{\pp^{(n)}}_{\infty}  $$
where the right-hand side converges to 0 by the hypothesis.
\end{proof}

\begin{lemma}\label{LemmeTechnique2}
If $ \displaystyle \norm{\pp^{(n)}}_{\infty}  \underset{n \rightarrow +\infty}{\longrightarrow} 0$, then for any $k \geq 1, m \geq 1$, 
$$ \sum_{\underset{\text{pairwise distinct}}{\ii_1, \dots ,\ii_m \in \widetilde{Q}_{k}}} p_{\ii_1} \dots p_{\ii_m} \underset{n \rightarrow +\infty}{\longrightarrow} \frac{1}{k^m}.     $$
\end{lemma}

\begin{proof}
We have
\begin{multline*}
1-k^m\sum_{\underset{\text{pairwise distinct}}{\ii_1, \dots ,\ii_m \in \widetilde{Q}_{k}}} p_{\ii_1} \dots p_{\ii_m} 
= \sum_{\ii_1, \dots ,\ii_m \in \I^k} p_{\ii_1} \dots p_{\ii_m} - \sum_{\underset{\text{pairwise not conjugate}}{\ii_1, \dots ,\ii_m \in Q_{k}}} p_{\ii_1} \dots p_{\ii_m}\\
= \sum_{(\ii_1, \dots, \ii_m) \in B_k} p_{\ii_1} \dots p_{\ii_m}
\end{multline*}
where $B_k = \left(\I^k\right)^m \setminus \left\{ \ii_1, \dots, \ii_m \in Q_k \text{ pairwise not conjugate} \right\}$. A tuple $\ii_1, \dots, \ii_m \in \left(\I^k\right)^m$ is in $B_k$ if and only if there exists $j$ such that $\ii_j$ is a power or is conjugate to at least one of the $\ii_s, s<j$. Define $$ B_k = \bigsqcup_{l = 1, \dots, m } B_k^j $$ where $B_k^l$ is the set of all elements of $B_k$ such that $j = l$ is the smallest $j$ such that $\ii_j$ is a power or is conjugate to at least one of the $\ii_s, s<j$. We then have $$ \sum_{(\ii_1, \dots, \ii_m) \in B_k} p_{\ii_1} \dots p_{\ii_m} = \sum_{l=1}^m \sum_{\ii_1, \dots \ii_m \in B_k^l} p_{\ii_1} \dots p_{\ii_m}.$$
The $l$-th term of this sum is $$ E_l := \sum_{\ii_1 \in Q_k} p_{\ii_1} \left( \sum_{\underset{\text{not conjugate to }\ii_1}{\ii_2 \in Q_{k}}} p_{\ii_2} \dots  \left(  \sum_{\underset{\text{or conjugate to }\ii_1 \text{ or...or }\ii_{l-1}}{\ii_l \text{ power}}} p_{\ii_l} \right) \right). $$ The innermost sum equals 1 since there is no condition on $i_{l+1}, \dots , i_m$. It can be rewritten by recalling that if $\ii$ and $\mathbf{j}$ are conjugates, then $p_{\ii} = p_{\mathbf{j}}$. We then get, by bounding the sums over $Q_k$ by the same sums over $\I^k$, that
$$ E_l \leq \sum_{\ii_1 \in \I^{k}} p_{\ii_1} \left( \sum_{\ii_2 \in \I^k} p_{\ii_2} \left( \dots \sum_{\ii_{l-1} \in \I^k} p_{\ii_{l-1}} \left(kp_{\ii_1} + \dots + kp_{\ii_{l-1}} + \sum_{\ii_l \text{ power}} p_{\ii_l}   \right) \right) \right) .$$
We rewrite this expression in the form 
\begin{multline*}
k \sum_{\ii_1 \in \I^{k}} p_{\ii_1}^2 + ... + k\sum_{\ii_{l-1} \in \I^{k}} p_{\ii_{l-1}}^2 + \sum_{\ii_l \text{ power}} p_{\ii_l} \leq k(l-1) \norm{\pp^{(n)}}_2^{2k} + \sum_{d=2, d|k}^k \sum_{\ii \in \I^{k/d}} p_{\ii}^d \\
= k(l-1) \norm{\pp^{(n)}}_2^{2k} + \sum_{d=2, d|k}^k \norm{\pp^{(n)}}_{d}^k.
\end{multline*}
By summing over $1 \leq l \leq m$, we then get that $$ 1 - k^m\sum_{\underset{\text{pairwise distinct}}{\ii_1, \dots ,\ii_m \in \widetilde{Q}_{k}}} p_{\ii_1} \dots p_{\ii_m} \leq k \frac{m(m-1)}{2} \norm{\pp^{(n)}}_2^{2k} + m \sum_{d=2, d|k}^k \norm{\pp^{(n)}}_{d}^k  $$
which tends to 0 when $n$ tends to infinity, by Lemma \ref{norme2tendvers0}, hence the result.
\end{proof}

We now turn to the proof of Theorem \ref{CyclesLoiVariable}.

\begin{proof}[Proof of Theorem \ref{CyclesLoiVariable}]
We establish the convergence of joint moments of the variables $c_j$. Let $1 \leq k_1, \dots ,k_p \leq k$ be pairwise distinct, and $r_1, \dots ,r_p \geq 1$. We have:
$$
\E\left(c_{k_1}^{r_1} \dots c_{k_p}^{r_p}\right)  = \E \left(  \left( \sum_{\ii \in \widetilde{Q}_{k_1}} D_{\ii} \right) ^{r_1} \dots \left(\sum_{\ii \in \widetilde{Q}_{k_p}} D_{\ii} \right) ^{r_p} \right).
$$
A partition $\{A_1, \dots, A_m \}$ of $\llbracket 1,r \rrbracket$ is the partition induced by $\ii_1, \dots, \ii_r$ if the sequence $\ii_1, \dots, \ii_r$ takes exactly $m$ distinct values, and if for any $s,t \in \llbracket 1,r \rrbracket$, $\ii_s = \ii_t$ if and only if $s$ and $t$ are in the same block in the partition $\{A_1, \dots, A_m \}$. We now expand each power, and we group the terms according to the induced partition of the indices:
$$\sum_{m_1=1}^{r_1} \dots \sum_{m_p=1}^{r_p}  \sum_{ \left\{ A_1^1, \dots ,A_{m_1}^1 \right\} } \dots \sum_{ \left\{ A_1^p, \dots ,A_{m_p}^p \right\} } \sum_{\underset{\text{pairwise distinct}}{\ii_1^1, \dots ,\ii_{m_1}^1 \in \widetilde{Q}_{k_1}}} \dots  \sum_{\underset{\text{pairwise distinct}}{\ii_1^p, \dots ,\ii_{m_p}^p \in \widetilde{Q}_{k_p}}} \E\left( D_{\ii_1^1}^{\#A_1^1} \dots D_{\ii_{m_p}^p}^{\#A_{m_p}^p}  \right) $$
where the sums over the $A_1^i, \dots ,A_{m_i}^i$ are over the partitions of $\llbracket 1,r_i \rrbracket$ in $m_i$ unordered nonempty subsets.\\
By Lemma \ref{LemmeTechnique1} and Lemma \ref{LemmeTechnique2}, we then get that
\begin{align*}
 \E\left(c_{k_1}^{r_1} \dots c_{k_p}^{r_p}\right) \underset{n \rightarrow +\infty}{\longrightarrow} & \sum_{m_1=1}^{r_1} \dots \sum_{m_p=1}^{r_p} \frac{1}{k_1^{m_1} \dots k_p^{m_p}} \sum_{\{A_1^1, \dots ,A_{m_1}^1\}} \dots \sum_{\{A_1^p, \dots ,A_{m_p}^p\}} 1 \\
 &= \sum_{m_1=1}^{r_1} \dots \sum_{m_p=1}^{r_p}  \frac{1}{k_1^{m_1} \dots k_p^{m_p}} S(r_1,m_1) \dots S(r_p,m_p) \\
 & = \left(\sum_{m_1=1}^{r_1} \frac{1}{k_1^{m_1}}S(r_1,m_1)\right) \dots \left(\sum_{m_p=1}^{r_p} \frac{1}{k_p^{m_p}}S(r_p,m_p)\right)
\end{align*}
where we recall that $S(r_i,m_i)$ is the Stirling number of the second kind, corresponding to the number of partitions in $m_i$ unordered nonempty subsets of $\llbracket 1,r_i \rrbracket $.\\
We recognize the product of the $r_i$th moments of Poisson distributions with parameters $\frac{1}{k_i}$, which concludes the proof.
\end{proof}

We can directly apply these result to the particular case of major-index-biased permutations.

\begin{proposition}
For $q_n = 1-o(1)$, and $\sigma_n$ following the distribution $\mathrm{Maj}(n,q_n)$, we have 
$$ \left(c_1, \dots ,c_k \right)  \overset{(d)}{\underset{n \rightarrow +\infty}{\longrightarrow}} \mathcal{P} \left( 1 \right) \otimes  \dots  \otimes \mathcal{P} \left(  \frac{1}{k} \right) . $$
\end{proposition}

Note that this result differs from the one where $q$ is fixed (see Corollary \ref{CVloipetitscyclesloifixe}), as was also noted in \cite[Section 6]{arXiv:2501.12513} for the fixed points.

\section{Asymptotic distribution of large cycles}\label{sectiongdscycles}

\subsection{Infinite-dimensional simplex and large cycles}

Let us denote  the infinite-dimensional standard simplex by $\displaystyle \Delta^{\infty} = \left\{ x \in [0,1]^{\N^*} \;\big|\; x_1 \geq x_2 \geq  \dots  \text{ and } \sum_{j \in \ \N^*} x_j \leq 1 \right\} $. In this section, we study a few convergence properties in this set, which we equip with the topology of pointwise convergence.

\begin{proposition}
The simplex $\Delta^{\infty}$ is compact.
\end{proposition}

\begin{proof}
First, it follows from Tykhonov's theorem that $[0,1]^{\N^*}$ is compact. The non-increasing sequences of $[0,1]^{\N^*}$ form a closed subset of $[0,1]^{\N^*}$. Indeed, we can write $$\displaystyle \left\{ x \in [0,1]^{\N^*} \;\big|\; x_1 \geq x_2 \geq  \dots  \right\}  = \bigcap_{j \geq 1} \left\{ x \in [0,1]^{\N^*} \;\big|\; x_j \geq x_{j+1} \right\} $$ which is an intersection of closed sets. Now, if $x^{(n)}$ is a sequence of elements of $\Delta^{\infty}$ which converges pointwise to $x$, Fatou's Lemma gives us $$ \sum_j x_j \leq \liminf_{n \rightarrow +\infty} \sum_j x_j^{(n)} \leq 1, $$
thus $x \in \Delta^{\infty}$. Hence, $\Delta^{\infty}$ is a closed subset of the compact $[0,1]^{\N^*}$, and the result follows.
\end{proof}

For $t \geq 2$, define
$$\begin{array}{ccccc}
 m_t & : & \Delta^{\infty} & \to & \R\\
& & x & \mapsto & \displaystyle \sum_j x_j^t
\end{array}$$\\
For $t=1$, we define $m_1$ identically equal to 1.

\begin{proposition}\label{propmt}
\begin{itemize}
\item For $a$ and $b$ in $\Delta^{\infty}$, if for any $t \geq 2,  m_t(a) = m_t(b)$, then $ a=b.   $
\item The $m_t$ for $t \geq 2$ are continuous.
\item For $a_1, a_2, \dots , a$ in $\Delta^{\infty}$, if for any $t \geq 2, m_t(a_n) \underset{n \rightarrow +\infty}{\longrightarrow} m_t(a)$ then $a_n \underset{n \rightarrow +\infty}{\longrightarrow} a$.
\end{itemize}
\end{proposition}

\begin{proof}
For the first point, note that $m_t(a)$ is the $(t-1)$th moment of the probability measure $\nu_{a} = \sum a_j \delta_{a_j} + \left(1-\sum a_j\right)\delta_0$, which has bounded support, hence, it is determined by its moments. Thus, if for any $t \geq 2,  m_t(a) = m_t(b)$, then $ \nu_{a} = \nu_{b}$. Now, the measure $\nu_a$ characterizes $a$: the greatest value of the sequence $a$ is the greatest value $a_1$ of the support, and appears $\nu_a(a_1)/a_1$ times in $a$, and so on. \\
For the second point, note that, for $x \in \Delta^{\infty}$, we have that for any $j \geq 1, x_j \leq \frac{1}{j}$, since the $x_j$ are non-increasing and their sum is 1. Thus, if $a_n \underset{n \rightarrow +\infty}{\longrightarrow} a$, then $m_t(a_n) \underset{n \rightarrow +\infty}{\longrightarrow} m_t(a)$ by the dominated convergence theorem, with $x = \frac{1}{j^t}$ as a dominating sequence.\\
The third point follows from the two previous ones, and the compactness of $\Delta^{\infty}$.
\end{proof}

\begin{proposition}\label{CVloisimplexe}
Let $X_1, X_2, \dots ,Z$ be random variables taking values in $\Delta^{\infty}$. Then $X_n$ converges in distribution in $Z$ if and only if for any $r \in \N^*$, for any $t_1, \dots ,t_r \geq 2$, we have $$\E\left(m_{t_1}(X_n) \dots m_{t_r}(X_n)\right) \underset{n \rightarrow +\infty}{\longrightarrow} \E(m_{t_1}(Z) \dots m_{t_r}(Z)).$$
\end{proposition}

\begin{proof}
Denote by $\mathcal{A}$ the subalgebra of $C(\Delta^{\infty})$ (the set of continuous (bounded) functions on $\Delta^{\infty}$) generated by the $m_t$ for $t \geq 1$. The direct implication follows immediately from the fact that the functions of $\mathcal{A}$ are continuous and bounded. Let us now prove the converse implication. By hypothesis, we have $$ \forall g \in \mathcal{A}, \E \left( g(X_n) \right) \underset{n \rightarrow +\infty}{\longrightarrow} \E \left( g(Z) \right). $$
The algebra $ \mathcal{A}$ contains the constant functions, and separates the points of $\Delta^{\infty}$ (Proposition \ref{propmt}), hence, Stone-Weierstrass theorem gives us that it is dense in $C_b(\Delta^{\infty})$.\\
Let $f \in C_b(\Delta^{\infty})$, let $\epsilon > 0$. Then there exists a function $g$ in $\mathcal{A}$ such that $\norm{g - f}_{\infty}< \frac{\epsilon}{2}$. Thus, 
\begin{align*}
\left| \E(f(X_n)) - \E(f(Z))  \right| & \leq \left| \E(f(X_n)) - \E(g(X_n))  \right| + \left| \E(g(X_n)) - \E(g(Z))  \right| +\left| \E(g(Z)) - \E(f(Z))  \right| \\
& \leq \norm{g - f}_{\infty} + \left| \E(g(X_n)) - \E(g(Z))  \right| + \norm{g - f}_{\infty} \\
& \leq \dfrac{\epsilon}{2} + \left| \E(g(X_n)) - \E(g(Z))  \right| + \dfrac{\epsilon}{2}.
\end{align*}
Since $g \in \mathcal{A}$, the second term tends to 0. Therefore, $$\forall \epsilon >0,  \limsup_{n \rightarrow +\infty} \left| \E(f(X_n)) - \E(f(Z))  \right|  \leq \epsilon. $$
This proves the result.
\end{proof}

\subsection{Uniform model and Poisson-Dirichlet process}

In the uniform model, the asymptotic distribution of the vector $(\lambda_1^{(n)},\lambda_2^{(n)}, \dots )$, where $\lambda_i^{(n)}$ is the length of the $i$th greatest cycle of a uniform random permutation of size $n$, is well known (Goncharov, 1944).

\begin{definition}
The Poisson-Dirichlet process (with parameter 1) is a random variable $Z$ taking values in $\Delta^{\infty}$, defined by the following procedure ("Stick-breaking process"):
\begin{itemize}
\item Consider i.i.d.~random variables $(U_k)_{k \geq 1}$ uniform in $(0;1)$;
\item define $Y_1 = U_1$ and $R_1 = 1-U_1$ (break a stick of length 1 and keep a piece of length $R_1$);
\item define $Y_2 = R_1U_2$ and $R_2 = R_1(1-U_2)$ (break a stick of length $R_1$ and keep a piece of length $R_2$);
\item iterate infinitely many times;
\item order the sequence $Y$ in decreasing order to get $Z$.
\end{itemize}
\end{definition}

See Figure \ref{stickbreaking} for a visual representation of the procedure.

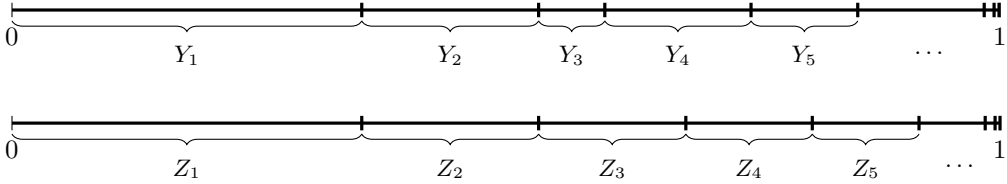
\begin{figure}[!h]
\centering

\begin{tikzpicture}[x=13cm,y=1.2cm]
  \tikzstyle{cut}=[very thick]
  \tikzstyle{tick}=[thin]
  \tikzstyle{lbl}=[font=\small]

  \draw[cut] (0,0) -- (1,0);
  \draw[tick] (0,-0.08) -- (0,0.08) node[below=6pt]{\(0\)};
  \draw[tick] (1,-0.08) -- (1,0.08) node[below=6pt]{\(1\)};

  \def\Llist{0.354,0.179,0.067,0.148,0.108,0.128,0.01,0.005}  

  \pgfmathsetmacro{\pos}{0}
  \foreach \L [count=\i] in \Llist {
    \pgfmathsetmacro{\next}{\pos+\L}
    \draw[cut] (\next,-0.08) -- (\next,0.08);
    \ifnum\i<6
      \draw[decorate,decoration={brace,amplitude=4pt}] (\next,-0.12) -- (\pos,-0.12)
        node[midway,below=6pt,lbl]{\(Y_\i\)};
    \fi
    \ifnum\i=5
      \pgfmathsetmacro{\remainingStart}{\next}
      \pgfmathsetmacro{\remainingEnd}{1}
      \pgfmathsetmacro{\middle}{(\remainingStart+\remainingEnd)/2}
      \node[below=6pt,font=\small] at (\middle,-0.18) {\(\dots\)};
    \fi
    \xdef\pos{\next}
}
\end{tikzpicture}
\\[0.5cm]

\begin{tikzpicture}[x=13cm,y=1.2cm]
  \tikzstyle{cut}=[very thick]
  \tikzstyle{tick}=[thin]
  \tikzstyle{lbl}=[font=\small]

  \draw[cut] (0,0) -- (1,0);
  \draw[tick] (0,-0.08) -- (0,0.08) node[below=6pt]{0};
  \draw[tick] (1,-0.08) -- (1,0.08) node[below=6pt]{1};

  \def\Llist{0.354,0.179,0.149,0.128,0.108,0.067,0.010,0.005}

  \pgfmathsetmacro{\pos}{0}
  \foreach \L [count=\i] in \Llist {
    \pgfmathsetmacro{\next}{\pos+\L}
    \draw[cut] (\next,-0.08) -- (\next,0.08);
    \ifnum\i<6
      \draw[decorate,decoration={brace,amplitude=4pt}] (\next,-0.12) -- (\pos,-0.12)
        node[midway,below=6pt,lbl]{\(Z_\i\)};
    \fi
    \ifnum\i=5
      \pgfmathsetmacro{\remainingStart}{\next}
      \pgfmathsetmacro{\remainingEnd}{1}
      \pgfmathsetmacro{\middle}{(\remainingStart+\remainingEnd)/2}
      \node[below=6pt,font=\small] at (\middle,-0.18) {\(\dots\)};
    \fi
    \xdef\pos{\next}
  }

\end{tikzpicture}

\caption{A visual representation of the Poisson-Dirichlet process.}
\label{stickbreaking}
\end{figure}

\begin{theorem}(e.g.~\cite[formula (1.36)]{ArratiaBarbourTavare})\label{CVPDUnif} The vector $\dfrac{\lambda^{(n)}}{n} = n^{-1}(\lambda_1^{(n)},\lambda_2^{(n)}, \dots )$, where $\lambda_j^{(n)}$ is the length of the $j$th greatest cycle of a uniform random permutation of size $n$, converges in distribution in $\Delta^{\infty}$ to the Poisson-Dirichlet process.
\end{theorem}

The idea behind this theorem is that, for a uniform random permutation $\sigma_n$ of size $n$, if $\rho$ is the cycle of $\sigma_n$ containing 1, the length of $\rho$ is uniform on $\left\{ 1, \dots , n \right\} $ and the restriction of $\sigma_n$ on $\llbracket 1,n \rrbracket \setminus \mathrm{supp}(\rho)$ is uniform on $\mathfrak{S}_{n-k}$, conditioned to $\# \mathrm{supp}(\rho) = k$. This is not the case for $\std(G)$, but the result for the uniform model will allow us to recover the limit distribution for the standardized model. To achieve this, we will need the values of $\E(m_{t_1}(Z) \dots m_{t_r}(Z))$, in the case where $Z$ is a Poisson-Dirichlet process, in order to apply Proposition \ref{CVloisimplexe}. These values are not easy to compute directly from the definition of the Poisson-Dirichlet process. However, since we know precisely the joint moments of cycles in the uniform model, we can use Proposition \ref{CVloisimplexe} and Theorem \ref{CVPDUnif} to get the values of interest for the Poisson-Dirichlet process.

\begin{lemma}\label{momentsProcPD}
Let $m \geq 1$, let $t_1, \dots ,t_r \geq 2$, let $Z$ be the Poisson-Dirichlet process. Then,
$$ \E(m_{t_1}(Z) \dots m_{t_r}(Z)) = \sum_{m=1}^r \sum_{\left\{ A_1, \dots ,A_m \right\} } \int_0^1 \dots \int_0^1 x_1^{t_{A_1}-1} \dots x_m^{t_{A_m}-1}\mathds{1}_{\sum x_i \leq 1} \mathrm{d}x_1 \dots \mathrm{d}x_m $$
where the second sum is over all the partitions of $\llbracket 1,r \rrbracket$ in $m$ unordered nonempty subsets, and where $t_{A} = \displaystyle \sum_{i \in A}t_i$.
\end{lemma}

To prove this Lemma, we will use the result from Theorem \ref{CVPDUnif} and the following lemma:

\begin{lemma}\cite[Lemma 1.1]{ArratiaBarbourTavare}\\
Let us denote $x^{[r]} = x(x-1) \dots (x-r+1)$. Then, for $l_1, \dots ,l_n \geq 1$,  $$ \E \left( \prod_{j=1}^n \left(C_j^{(n)}\right)^{[l_1]}  \right) = \left(   \prod_{j=1}^n \frac{1}{j^{l_j}} \right) \mathds{1}\left\{ \sum_{j=1}^n jl_j \leq n   \right\}  $$
where the $C_j^{(n)}$ are the number of cycles of length $j$ in a uniform random permutation of size $n$.
\end{lemma}

From this lemma, it follows the result:

\begin{lemma}\label{momentscyclesunif}
Let $k_1, \dots ,k_s \geq 1$ be pairwise distinct, and let $r_1, \dots ,r_s \geq 1$. Then,  $$ \E\left( \left(C_{k_1}^{(n)}\right)^{r_1} \cdots \left(C_{k_s}^{(n)}\right)^{r_s}  \right) = \sum_{m_1=1}^{r_1} \dots \sum_{m_s=1}^{r_s} \frac{1}{k_1^{m_1} \dots k_s^{m_s}} S(r_1,m_1) \dots S(r_s,m_s) \mathds{1}_{\sum k_jm_j \leq n} $$
where $S(r_i,m_i)$ is the Stirling number of the second kind, corresponding to the number of partitions in $m_i$ unordered nonempty subsets of $\llbracket 1,r_i \rrbracket $.
\end{lemma}

\begin{proof}
This follows from the polynomial formula $$ X^r = \sum_{m=0}^r S(r,m)X^{[m]} $$ in which the Stirling numbers of the second kind appear.
\end{proof}

We can now prove Lemma \ref{momentsProcPD}.

\begin{proof}[Proof of Lemma \ref{momentsProcPD}]
We use Lemma \ref{momentscyclesunif} to get the limit of $\E\left(m_{t_1}\left(\dfrac{\lambda^{(n)}}{n}\right) \dots m_{t_r}\left(\dfrac{\lambda^{(n)}}{n}\right)\right)$, which is the required value for the Poisson-Dirichlet process, since Proposition \ref{CVloisimplexe} and Theorem  \ref{CVPDUnif} hold.\\
Remark that we have $$ \lambda = ( \underbrace{n, \dots ,n}_{C_n \text{ times}},\underbrace{n-1, \dots ,n-1}_{C_{n-1} \text{ times}}, \dots ,\underbrace{j, \dots ,j}_{C_j \text{ times}}, \dots    ) .$$
We then have for $t_1, \dots ,t_r \geq 2$:
\begin{align*}
\E\left(m_{t_1}\left(\dfrac{\lambda^{(n)}}{n}\right) \dots m_{t_r}\left(\dfrac{\lambda^{(n)}}{n}\right)\right) & = \E \left( \left( \sum_{j \geq 1} \frac{\lambda_j^{t_{1}}}{n^{t_1}} \right) \cdots \left( \sum_{j \geq 1} \frac{\lambda_j^{t_{r}}}{n^{t_r}} \right) \right)\\
 & = \E \left( \left( \frac{1}{n^{t_1}} \sum_{k \geq 1} C_kk^{t_1} \right) \cdots \left( \frac{1}{n^{t_r}} \sum_{k \geq 1} C_kk^{t_r} \right) \right)\\
 & = \sum_{k_1, \dots ,k_r \geq 1} \left(\frac{k_1}{n}\right)^{t_1} \cdots \left(\frac{k_r}{n}\right)^{t_r} \E(C_{k_1} \dots C_{k_r}).
\end{align*}
We now group the terms according to the partition induced by $\{ k_1, \dots, k_r \}$ to apply Lemma \ref{momentscyclesunif}:

$\quad \E\left(m_{t_1}\left(\dfrac{\lambda^{(n)}}{n}\right) \dots m_{t_r}\left(\dfrac{\lambda^{(n)}}{n}\right)\right) $

$ \displaystyle=  \sum_{m=1}^r \sum_{\underset{\text{partition of }\llbracket 1,r \rrbracket}{\left\{ A_1, \dots ,A_m \right\} }} \sum_{\underset{\text{pairwise distinct}}{k_1, \dots ,k_m}} \left(\frac{k_1}{n}\right)^{t_{A_1}} \cdots \left(\frac{k_m}{n}\right)^{t_{A_m}} \E(C_{k_1}^{\#A_1} \dots C_{k_m}^{\#A_m})$

\begin{multline*}
= \sum_{m=1}^r \sum_{\underset{\text{partition of }\llbracket 1,r \rrbracket}{\left\{ A_1, \dots ,A_m \right\} }} \sum_{\underset{\text{pairwise distinct}}{k_1, \dots ,k_m}} \left(\frac{k_1}{n}\right)^{t_{A_1}} \cdots \left(\frac{k_m}{n}\right)^{t_{A_m}} \sum_{l_1=1}^{\#A_1} \dots \\ \sum_{l_m=1}^{\#A_m} \frac{1}{k_1^{l_1} \dots k_s^{l_m}} S(\#A_1,l_1) \dots S(\#A_m,l_m) \mathds{1}_{\sum k_jl_j \leq n}.
\end{multline*}

If $l_1= \dots =l_m = 1$, the term 
\begin{multline*}
\sum_{\underset{\text{pairwise distinct}}{k_1, \dots ,k_m}} \left(\frac{k_1}{n}\right)^{t_{A_1}} \cdots \left(\frac{k_m}{n}\right)^{t_{A_m}} \frac{1}{k_1 \dots k_s}  \mathds{1}_{\sum k_j \leq n} \\ = \frac{1}{n^m}\sum_{\underset{\text{pairwise distinct}}{k_1, \dots ,k_m}} \left(\frac{k_1}{n}\right)^{t_{A_1}-1} \cdots \left(\frac{k_m}{n}\right)^{t_{A_m}-1}  \mathds{1}_{\sum k_j \leq n}
\end{multline*}
converges as $n$ tends to infinity to the limit:
$$\int_0^1 \dots \int_0^1 x_1^{t_{A_1}-1} \dots x_1^{t_{A_m}-1}\mathds{1}_{\sum x_j \leq 1} \mathrm{d}x_1 \dots \mathrm{d}x_m $$ since we recognize a multidimensional Riemann sum (the condition over the $k_i$ does not affect the result: we only remove a number $O\left(n^{m-1}\right)$ of terms, which leads to an error in $O\left(n^{-1}\right)$).\\
If the $l_i$ are not all 1, then the term is 

\begin{multline*} \sum_{\underset{\text{pairwise distinct}}{k_1, \dots ,k_m}} \left(\frac{k_1}{n}\right)^{t_{A_1}} \cdots \left(\frac{k_m}{n}\right)^{t_{A_m}} \frac{1}{k_1^{l_1} \dots k_s^{l_m}} S(\#A_1,l_1) \dots S(\#A_m,l_m) \mathds{1}_{\sum k_jl_j \leq n} \\= \frac{1}{n^{\sum t_j}} \sum_{\underset{\text{pairwise distinct}}{k_1, \dots ,k_m}} S(\#A_1,l_1) \dots S(\#A_m,l_m) \left(\frac{k_1}{n}\right)^{t_{A_1}-l_1} \cdots \left(\frac{k_m}{n}\right)^{t_{A_m}-l_m}  \mathds{1}_{\sum k_jl_j \leq n}  
\end{multline*}
and thus is equivalent to $\displaystyle \frac{1}{n^{\sum t_j -m}} \int_0^1 \dots \int_0^1 x_1^{t_{A_1}-l_1} \dots x_1^{t_{A_m}-l_m}\mathds{1}_{\sum x_jl_j \leq 1} \mathrm{d}x_1 \dots \mathrm{d}x_m $ which tends to 0.\\
We therefore have that $ \E\left(m_{t_1}\left(\dfrac{\lambda^{(n)}}{n}\right) \dots m_{t_r}\left(\dfrac{\lambda^{(n)}}{n}\right)\right)$ converges to $$ \sum_{m=1}^r \sum_{\left\{ A_1, \dots ,A_m \right\} } \int_0^1 \dots \int_0^1 x_1^{t_{A_1}-1} \dots x_1^{t_{A_m}-1}\mathds{1}_{\sum x_j \leq 1} \mathrm{d}x_1 \dots \mathrm{d}x_m   $$ which is indeed the claimed value for the Poisson-Dirichlet process.
\end{proof}

\subsection{Large cycles of studied model}

We return to the setting of the standardized permutation model, with $\pp$ depending on $n$, and we assume that there exists $R \in (0;1)$ such that for any $n \geq 1$ and $i \in \I, p_i \leq R$. Then, we extend the convergence to the Poisson-Dirichlet process from the uniform model to the standardized permutation model, i.e.~we prove Theorem \ref{CVPDQS}. We first establish a few technical lemmas to prove the theorem.

\begin{lemma}\label{momentsjointsdi}
Let $\epsilon >0$, let $k_1, \dots ,k_m \geq \epsilon n$, let $\ii_1 \in Q_{k_1}, \dots ,\ii_m \in Q_{k_m}$ pairwise non-conjugate, let $r \geq 1$ and let $s_1, \dots ,s_m \geq 1$ with sum less than or equal to $r$. Then, $$ \E \left( D_{\ii_1}^{s_1} \dots D_{\ii_m}^{s_m}   \right) = p_{\ii_1} \dots p_{\ii_m} \mathds{1}_{\sum k_j \leq n} \left( 1 + O\left( R^{\epsilon n} \right) \right)  $$ where $O$ only depends on $r$ and $m$ (and $R$ which is supposed to be fixed in this section).
\end{lemma}

\begin{proof}
By Lemma \ref{momentsFctQueue} and Theorem \ref{loiDiCycles},
\begin{align*}
\E \left( D_{\ii_1}^{s_1} \dots D_{\ii_m}^{s_m}   \right) & = \sum_{\underset{\sum k_jl_j \leq n}{l_1, \dots ,l_m \geq 1}} p_{\ii_1}^{l_1} \dots p_{\ii_m}^{l_m} \prod_{j=1}^m \left( l_j^{s_j} - (l_j -1)^{s_j} \right)\\
& = p_{\ii_1} \dots p_{\ii_m} \mathds{1}_{\sum k_i \leq n} \sum_{\underset{\sum k_jl_j \leq n - \sum k_j}{l_1, \dots ,l_m \geq 0}} p_{\ii_1}^{l_1} \dots p_{\ii_m}^{l_m} \prod_{j=1}^m \left( (l_j+1)^{s_j} - l_j^{s_j} \right).
\end{align*}
Now control the sum. Recall that for any $j, p_{\ii_j} \leq R^{k_j}$. Then,
\begin{multline*}
 \sum_{\underset{\sum k_jl_j \leq n - \sum k_j}{l_1, \dots ,l_m \geq 0}} p_{\ii_1}^{l_1} \dots p_{\ii_m}^{l_m} \prod_{j=1}^m \left( (l_j+1)^{s_j} - l_j^{s_j} \right) - 1 
\leq  \sum_{\underset{\text{not all zero}}{l_1, \dots ,l_m \geq 0}} R ^{\sum k_jl_j} \prod_{j=1}^m \left( (l_j+1)^{s_j} - l_j^{s_j} \right)\\
\leq  R ^{\epsilon n} m \left(\sum_{l > 0} R ^{l-1} \left( (l+1)^r - l^r \right) \right) \left(\sum_{l \geq 0} R ^{l} \left( (l+1)^r - l^r \right) \right)^{m-1}
=  R^{\epsilon n}O(1)
\end{multline*}
and hence the result holds.
\end{proof}

\begin{lemma}\label{momentscyclesmacro}
Let $\epsilon >0$, let $k_1, \dots ,k_s > \epsilon n$ be pairwise distinct, let $r_1, \dots ,r_s \geq 1$ with sum less than or equal to $r$. Then,  $$ \E\left( c_{k_1}^{r_1} \dots c_{k_s}^{r_s}  \right) = \left(\sum_{m_1=1}^{r_1} \dots \sum_{m_s=1}^{r_s} \frac{1}{k_1^{m_1} \dots k_s^{m_s}} S(r_1,m_1) \dots S(r_s,m_s) \mathds{1}_{\sum k_jm_j \leq n} \right) \left( 1+O\left( R^{\epsilon n/4} \right) \right) $$
where the $O$ only depends on $r$ (and $R$ which is supposed to be fixed).
\end{lemma}

Note that the joint moments of the cycles are asymptotically the same as for the uniform model for the "macroscopic" cycles (with size greater than $\epsilon n$). What remains to be shown is that only these cycles contribute to the limiting expectations of Lemma \ref{momentsProcPD}.

\begin{proof}
Firstly we have $$ \E\left( c_{k_1}^{r_1} \dots c_{k_s}^{r_s}  \right) = \E\left( \left( \sum_{\ii_1 \in \widetilde{Q}_{k_1}} D_{\ii_1} \right)^{r_1} \dots \left( \sum_{\ii_s \in \widetilde{Q}_{k_s}} D_{\ii_s} \right)^{r_s}  \right)  . $$
We expand this expression, and we group the terms according to the partition induced by the indices:
\begin{align*}
& \E\left( \left( \sum_{\ii_1 \in \widetilde{Q}_{k_1}} D_{\ii_1} \right)^{r_1} \dots \left( \sum_{\ii_s \in \widetilde{Q}_{k_s}} D_{\ii_s} \right)^{r_s}  \right)\\
= & \sum_{m_1=1}^{r_1} \dots \sum_{m_s=1}^{r_s}   \sum_{\underset{\text{partition of }\llbracket 1, r_1 \rrbracket}{\left\{ A_1^1, \dots ,A_{m_1}^1 \right\}  }} \dots \sum_{\underset{\text{partition of }\llbracket 1, r_s \rrbracket}{ \left\{ A_1^s, \dots ,A_{m_s}^s \right\} }}     \sum_{\underset{\text{pairwise }\neq}{\ii_1^1, \dots ,\ii_{m_1}^1 \in \widetilde{Q}_{k_1}}} \dots \sum_{\underset{\text{pairwise }\neq}{\ii_1^s, \dots ,\ii_{m_s}^s \in \widetilde{Q}_{k_s}}}  \E\left( D_{\ii_1^1}^{\#A_1^1} \dots D_{\ii_{m_s}^s}^{\#A_{m_s}^s}  \right)
\end{align*}
\begin{multline*}
=\left( 1 + O\left( R^{\epsilon n} \right) \right) \sum_{m_1=1}^{r_1} \dots \sum_{m_s=1}^{r_s} S(r_1,m_1) \dots S(r_s,m_s) \sum_{\underset{\text{pairwise }\neq}{\ii_1^1, \dots ,\ii_{m_1}^1 \in \widetilde{Q}_{k_1}}} \dots \\ \sum_{\underset{\text{pairwise }\neq}{\ii_1^s, \dots ,\ii_{m_s}^s \in \widetilde{Q}_{k_s}}}  p_{\ii_1^1} \dots p_{\ii_{m_s}^s} \mathds{1}_{\sum k_jm_j \leq n}
\end{multline*}
by Lemma \ref{momentsjointsdi} with $r = \sum r_i$. Using the calculation from the proof of Lemma \ref{LemmeTechnique2}, we note that $$\displaystyle \sum_{\underset{\text{pairwise }\neq}{\ii_1^1, \dots ,\ii_{m_1}^1 \in \widetilde{Q}_{k_1}}} \dots \sum_{\underset{\text{pairwise }\neq}{\ii_1^s, \dots ,\ii_{m_s}^s \in \widetilde{Q}_{k_s}}}  p_{\ii_1^1} \dots p_{\ii_{m_s}^s} \mathds{1}_{\sum k_jm_j \leq n}$$ can be replaced by $\dfrac{1}{k_1^{m_1} \dots k_s^{m_s}} $ with an error $ O\left( R^{\epsilon n/4} \right) $ (with the inequality $\norm{\pp}_2 \leq \sqrt{\norm{\pp}_{\infty}}$). The result follows.
\end{proof}

We can now prove Theorem \ref{CVPDQS}.
\begin{proof}[Proof of Theorem \ref{CVPDQS}]
We show that the moments $\E\left(m_{t_1}\left(\dfrac{\lambda^{(n)}}{n}\right) \dots m_{t_r}\left(\dfrac{\lambda^{(n)}}{n}\right)\right)$ converge to the same value as in the uniform model, which will allow us to conclude using Proposition \ref{CVloisimplexe}.\\
Note that we have $$ \lambda^{(n)} = ( \underbrace{n, \dots ,n}_{C_n \text{ times}},\underbrace{n-1, \dots ,n-1}_{C_{n-1} \text{ times}}, \dots ,\underbrace{j, \dots ,j}_{C_j \text{ times}}, \dots    ) .$$
Thus,
$$\E\left(m_{t_1}\left(\dfrac{\lambda^{(n)}}{n}\right) \dots m_{t_r}\left(\dfrac{\lambda^{(n)}}{n}\right)\right) = \sum_{k_1, \dots ,k_r \geq 1} \left(\frac{k_1}{n}\right)^{t_1} \dots \left(\frac{k_r}{n}\right)^{t_r} \E(c_{k_1} \dots c_{k_r}).$$
Define the function $f_n$ on $(0,1]$ by $$ f_n(\epsilon) = \sum_{k_1, \dots ,k_r > \epsilon n} \left(\frac{k_1}{n}\right)^{t_1} \dots \left(\frac{k_r}{n}\right)^{t_r} \E(c_{k_1} \dots c_{k_r}) .$$ 
We have the convergence $$f_n(\epsilon) \underset{\epsilon \rightarrow 0^+}{\longrightarrow} \E\left(m_{t_1}\left(\dfrac{\lambda^{(n)}}{n}\right) \dots m_{t_r}\left(\dfrac{\lambda^{(n)}}{n}\right)\right)$$ 
and this convergence is uniform in $n$. Indeed, 
\begin{align}
& \left| f_n(\epsilon) - \E\left(m_{t_1}\left(\dfrac{\lambda^{(n)}}{n}\right) \dots m_{t_r}\left(\dfrac{\lambda^{(n)}}{n}\right)\right) \right| \\
=& \sum_{k_1\leq \epsilon n \text{ or ... or }k_r\leq \epsilon n} \left(\frac{k_1}{n}\right)^{t_1} \dots \left(\frac{k_r}{n}\right)^{t_r} \E(c_{k_1} \dots c_{k_r})\\
\leq & \, \frac{1}{n_1^{t_1} \dots n_r^{t_r}} r \sum_{k_1 \leq \epsilon n} \sum_{k_2, \dots ,k_r \geq 1} \E\left( k_1^{t_1}c_{k_1} \dots k_r^{t_r}c_{k_r} \right)\\
\leq & \, \frac{1}{n_1^{t_1} \dots n_r^{t_r}} r \E \left( \left( \sum_{k_1 \leq \epsilon n} k_1^{t_1}c_{k_1} \right) \left( \sum_{k_2 \geq 1} k_2c_{k_2} \right)^{t_2} \dots \left( \sum_{k_r \geq 1} k_rc_{k_r} \right)^{t_r} \right) \label{ligne}\\
\leq  & \, \epsilon ^{t_1-1}r
\end{align}
where at step \ref{ligne}, we have bounded $ \sum k_j^{t_j}c_{k_j}$ for $j \geq 2$ by $\left( \sum k_jc_{k_j} \right)^{t_j} = n^{t_j}$, and we have bounded $\left( \sum_{k_1 \leq \epsilon n} k_1^{t_1}c_{k_1} \right)$ by $(\epsilon n)^{t_1-1} \sum k_1c_{k_1} = \epsilon^{t_1-1}n^{t_1}$. Now, by Lemma \ref{momentscyclesmacro}, for a fixed $\epsilon > 0$,

$\displaystyle f_n(\epsilon) = \sum_{k_1, \dots ,k_r > \epsilon n} \left(\frac{k_1}{n}\right)^{t_1} \dots \left(\frac{k_r}{n}\right)^{t_r} \E(c_{k_1} \dots c_{k_r})$

$ \qquad \, \, \, \displaystyle =  \, \sum_{m=1}^r \sum_{\underset{\text{partition of }\llbracket 1,r \rrbracket}{\left\{ A_1, \dots ,A_m \right\} }} \sum_{\underset{\text{pairwise distinct}}{k_1, \dots ,k_m > \epsilon n}} \left(\frac{k_1}{n}\right)^{t_{A_1}} \dots \left(\frac{k_m}{n}\right)^{t_{A_m}} \E(c_{k_1}^{\#A_1} \dots c_{k_m}^{\#A_m})$
\begin{multline*}
\qquad \displaystyle \underset{n \rightarrow +\infty}{\sim} \, \sum_{m=1}^r \sum_{\underset{\text{partition of }\llbracket 1,r \rrbracket}{\left\{ A_1, \dots ,A_m \right\} }} \sum_{\underset{\text{pairwise distinct}}{k_1, \dots ,k_m > \epsilon n}} \left(\frac{k_1}{n}\right)^{t_{A_1}} \dots \left(\frac{k_m}{n}\right)^{t_{A_m}} \sum_{l_1=1}^{\#A_1} \dots \\ \sum_{l_m=1}^{\#A_m} \frac{1}{k_1^{l_1} \dots k_s^{l_m}} S(\#A_1,l_1) \dots S(\#A_m,l_m) \mathds{1}_{\sum k_jl_j \leq n}
\end{multline*}

$\qquad \displaystyle \underset{n \rightarrow +\infty}{\sim} \, \sum_{m=1}^r \sum_{\left\{ A_1, \dots ,A_m \right\} } \int_{\epsilon}^1 \dots \int_{\epsilon}^1 x_1^{t_{A_1}-1} \dots x_1^{t_{A_m}-1}\mathds{1}_{\sum x_j \leq 1} \mathrm{d}x_1 \dots \mathrm{d}x_m  : = f(\epsilon)$

with the same argument as in the proof of Lemma \ref{momentsProcPD}.
Thus, by exchanging limits (as the convergence for $\epsilon \rightarrow 0$ is uniform in $n$), it holds 
\begin{align*}
& \lim_{n \rightarrow +\infty} \E\left(m_{t_1}\left(\dfrac{\lambda^{(n)}}{n}\right) \dots m_{t_r}\left(\dfrac{\lambda^{(n)}}{n}\right)\right)
 = \lim_{n \rightarrow +\infty} \lim_{\epsilon \rightarrow 0^+} f_n(\epsilon)
= \lim_{\epsilon \rightarrow 0^+} f(\epsilon)\\
 = & \lim_{\epsilon \rightarrow 0^+} \sum_{m=1}^r \sum_{\left\{ A_1, \dots ,A_m \right\} } \int_{\epsilon}^1 \dots \int_{\epsilon}^1 x_1^{t_{A_1}-1} \dots x_1^{t_{A_m}-1}\mathds{1}_{\sum x_j \leq 1} \mathrm{d}x_1 \dots \mathrm{d}x_m\\
= &\sum_{m=1}^r \sum_{\left\{ A_1, \dots ,A_m \right\} } \int_0^1 \dots \int_0^1 x_1^{t_{A_1}-1} \dots x_1^{t_{A_m}-1}\mathds{1}_{\sum x_j \leq 1} \mathrm{d}x_1 \dots \mathrm{d}x_m.
\end{align*}
We recover the expression of $\E(m_{t_1}(Z) \dots m_{t_r}(Z))$ of Lemma \ref{momentsProcPD}, where $Z$ is the Poisson-Dirichlet process, hence we have the stated convergence by Proposition \ref{CVloisimplexe}.
\end{proof}

\section{Cycle count normality}\label{Sectioncyclecountnormality}

This section deals with the total number of cycles of the permutation $\std(G)$. We recall that, for the uniform model, a central limit type theorem holds (proved in 1944 by Goncharov):

\begin{theorem}(see e.g.~\cite[(1.31)]{ArratiaBarbourTavare})
Let $K_n$ be the number of cycles of a uniform random permutation $\sigma_n \in \mathfrak{S}_n$. Then, as $n$ tends to infinity, $$ \frac{K_n - \log(n)}{\sqrt{\log(n)}} \overset{(d)}{\longrightarrow} \mathcal{N}(0,1)$$
and the moments converge.
\end{theorem}

Is is natural to ask what happens for the $\std(G)$ model. Actually, the same result holds (Theorem \ref{Cyclecountnormality}). For $n \geq 1$, define $K_n = \sum_{k=1}^n c_k^{(n)}$, the total number of cycles of $\std(G)$.  To prove this theorem, we will use the convergence of cumulants to those of the limit distribution. 

\begin{definition}
\begin{itemize}
\item The joint cumulant $\kappa$ of several random variables $X_1, \dots , X_r$ with finite moments is defined by $$ \kappa(X_1, \dots, X_r) = \frac{\partial ^r G}{\partial t_1 \dots \partial t_r} (0, \dots, 0) $$  where $G(t_1, \dots, t_r) = \log \E \left( e^{\sum_{j=1}^r t_jX_j} \right)$.
\item For $r \geq 1$, the $r$th cumulant of a random variable $X$ is defined by $$ \kappa_r(X) = \kappa(X, \dots, X) $$ where $X$ appears $r$ times in the joint cumulant $\kappa(X, \dots, X)$.
\end{itemize}

\end{definition}

The cumulants give us an alternative to the moments of a distribution: their convergence are equivalent but, in some cases, it can be more convenient to choose one or the other. The cumulants are a common and useful notion to study the convergence of a sequence of random variable to a normal limit. More precisely, the following common properties hold:

\begin{proposition}\label{cumulantformula}
\begin{itemize}
\item The map which maps $X_1, \dots, X_r$ to $\kappa(X_1, \dots, X_r)$ is $r$-linear.
\item If the family $(X_j)_{1 \leq j \leq r}$ of random variables can be split into two independent families, then $\kappa(X_1, \dots, X_r) = 0$.
\item Let $r \geq 1$ and let $X_1,\dots,X_r$ be random variables. Then, we have a relationship between the joint cumulants and the joint moments given by
$$ \kappa \left( X_1,...,X_r \right) = \sum_{\pi} \mu(\pi, \{[r]\}) \prod_{B \in \pi} \E \left( \prod_{j \in B} X_j \right)  $$ where the sum runs over all the partitions $\pi$ of the set $[r] = \{ 1,\dots,r \}$, and where $\mu(\pi, \{[r]\}) = (|\pi|-1)! (-1)^{|\pi|-1}$ is the Möbius function on the set of all the partitions of $[r]$.
\item In particular, $\kappa(X_1) = \E(X_1)$ and $\kappa(X_1,X_2) = \mathrm{Cov}(X_1,X_2)$.
\item Let $X_n$ be a sequence of random variable, then the moments converge if and only if the cumulants converge. In this case, if the limits of the cumulants are the cumulants of a moment-determinate distribution $\mu$, then $X_n \overset{(d)}{\longrightarrow} \mu$.
\item If $Z \sim \mathcal{N}(0,1)$, then, for any $r \geq 1$, $$ \kappa_r(Z) = \begin{cases}
  1 & \text{if } r = 2; \\
  0 & \text{else.}
\end{cases} $$
\end{itemize}
\end{proposition}

The aim of what follows is to prove that the $r$th cumulant of the random variable in Theorem \ref{Cyclecountnormality}, converges to 1 if $r=2$ and 0 otherwise, from which the result follows. To achieve this, we will also need the useful following formula, from Leonov and Shiryaev, giving a relation between the cumulants of random variables with repetition and the cumulants of power of random variables without repetition, which is more convenient in our case.

\begin{definition}
Let $r$ be a positive integer, and let $\pi$ and $\pi'$ be two partitions of $[r] = \{ 1,\dots,r \}$. Then, the join $\pi \vee \pi'$ of $\pi$ and $\pi'$ is the finest partition which is coarser than $\pi$ and $\pi'$.
\end{definition}

To get a different point of view, $j$ and $k$ are in the same block of  $\pi \vee \pi'$ if and only if there exists a chain $j_1, \dots , j_l$ such that $j_1 = j, j_l=k$ and for any $m \in \llbracket 1, l-1 \rrbracket$, $j_m$ and $j_{m+1}$ belong to the same block in $\pi$ or in $\pi'$.

\begin{theorem}(\cite{zbMATH03143189}, see also \cite[Theorem 4.4]{Sniady06})\label{LeonovShiryaev}
Let $r \geq 1$ and let $X_1,\dots,X_r$ be random variables. Let $\pi_0$ be a partition of $[r] = \{ 1,\dots,r \}$. Then, we have
$$ \kappa \left( \prod_{j \in B_0} X_j, B_0 \in \pi_0 \right) = \sum_{\pi \vee \pi_0 = \{ [r] \} } \prod_{B \in \pi} \kappa \left( X_j, j \in B \right) $$ where the sum runs over all the partitions $\pi$ of $[r]$ such that the join $\pi \vee \pi_0$ of the partitions $\pi$ and $\pi_0$ is $\{ [r] \}$.
\end{theorem}

Finally, since we will actually study the convergence in distribution of a modified variable, we need a criterion of convergence for the moments of the variable of interest.
 
\begin{proposition}\label{CVmoments}\cite[Corollary of Theorem 25.12]{BillPandM}
Let $r<s$ be two positive real numbers. Let $X_n$ be a sequence of random variable, and $Z$ be a random variable. Suppose that $X_n \overset{(d)}{\longrightarrow} Z$ and that $\E(|X_n|^s)$ is bounded. Then, $\E(X_n^r)$ converges to $\E(Z^r)$.
\end{proposition}

Let us get back to the model. We first show two useful lemmas before proving Theorem \ref{Cyclecountnormality}.

\begin{lemma}\label{powercumulants}
Let $R \in (0;1)$, and suppose that for any $n \geq 1$ and $i \in \I, p_i \leq R$. Let $2 \leq m \leq r$, let $k_1,...,k_m \leq n/ \log(\log(n))$, let $\ii_1 \in \widetilde{Q}_{k_1},...,\ii_m \in \widetilde{Q}_{k_m}$ be pairwise distinct, and let $q_1,...,q_m \geq 1$ be such that $\sum q_j = r$. Then, it holds
$$\kappa\left( D_{\ii_1}^{q_1},\dots,D_{\ii_m}^{q_m} \right) = \left( \prod_{j=1}^m \frac{p_{\ii_j}}{k_j} \right)O\left(R^{n/2}\right) $$ where O only depends on $r$.
\end{lemma}

\begin{proof}
By Proposition \ref{cumulantformula} and Theorem \ref{loiDiCycles},
\begin{align*}
\kappa\left( D_{\ii_1}^{q_1},\dots,D_{\ii_m}^{q_m} \right) &= \sum_{\pi} \mu(\pi, \{[r]\}) \prod_{B \in \pi} \E \left( \prod_{j \in B} D_{\ii_j}^{q_j} \right)\\
& = \sum_{\pi} \mu(\pi, \{[r]\}) \prod_{B \in \pi} \sum_{\underset{\sum k_jl_j \leq n}{l_j \geq 1, j\in B}} \prod_{j=1}^m \left[ l_j^{q_j} - (l_j-1)^{q_j} \right] p_{\ii_j}^{l_j}\\
& =  \sum_{\pi} \mu(\pi, \{[r]\}) \sum_{\underset{\forall B \in \pi, \sum_{B} k_jl_j \leq n}{l_1,\dots,l_m \geq 1}} \prod_{j=1}^m \left[ l_j^{q_j} - (l_j-1)^{q_j} \right] p_{\ii_j}^{l_j}
\end{align*}
The index sets all contain the set $\displaystyle \left\{ l_1,...,l_m \geq 1, \sum_{j=1}^m k_jl_j \leq n \right\}$. All the corresponding terms of the sum appear with coefficient $\sum_{\pi} \mu(\pi, \{[r]\}) = 0$. Hence, writing $C = \sum_{\pi} \left| \mu(\pi, \{[r]\}) \right|$, it holds
\begin{align*}
\left| \kappa\left( D_{\ii_1}^{q_1},\dots,D_{\ii_m}^{q_m} \right) \right| & \leq C \sum_{\underset{k_jl_j\leq n \forall j ; \sum k_jl_j > n}{l_1,...,l_m \geq 1}} \prod_{j=1}^m \left[ l_j^{q_j} - (l_j-1)^{q_j} \right] p_{\ii_j}^{l_j}\\
& \leq C \sum_{\underset{k_jl_j\leq n \forall j ; \sum k_jl_j > n}{l_1,...,l_m \geq 1}} \prod_{j=1}^m n^{q_j} p_{\ii_j} R ^{k_j(l_j-1)}\\
\end{align*}
where we use the inequalities $l^q - (l-1)^q \leq l^q \leq n^q$ for $l\leq n$, and $p_{\ii_j} \leq R^{k_j}$. Now, by bounding the number of indices by $\frac{n^m}{\prod k_j}$, 
\begin{align*}
\left| \kappa\left( D_{\ii_1}^{q_1},\dots,D_{\ii_m}^{q_m} \right) \right| & \leq C R^{n-\sum k_j}\sum_{\underset{k_jl_j\leq n \forall j ; \sum k_jl_j > n}{l_1,...,l_m \geq 1}} \prod_{j=1}^m n^{q_j} p_{\ii_j}\\
&\leq C R^{n-mn/\log(\log(n))} n^m \left( \prod_{j=1}^m \frac{p_{\ii_j}}{k_j} \right) n^r\\
& = \left( \prod_{j=1}^m \frac{p_{\ii_j}}{k_j} \right)O\left(R^{n/2}\right) . \qedhere
\end{align*} 
\end{proof}

\begin{lemma}\label{PrimNegl2}
Let $R \in (0;1)$, and suppose that for any $n \geq 1$ and $i \in \I, p_i \leq R$. Then, it holds $$ \sum_{\ii \in \widetilde{Q}_k} p_{\ii} \underset{k \rightarrow +\infty}{=} \frac{1}{k} \left( 1 + o(1) \right) $$ where $o(1)$ only depends on $R$.
\end{lemma}

\begin{proof}
We first have
\begin{multline*}
\left| 1- k \sum_{\ii \in \widetilde{Q}_k} p_{\ii} \right| = \sum_{\ii \in \I^k} p_{\ii} -\sum_{\ii \in Q_k} p_{\ii} 
= \sum_{\underset{\text{power}}{\ii \in \I^k}}  p_{\ii}
\leq \sum_{\underset{d | k}{d \geq 2}} \sum_{i_1,...,i^{k/d} \in \I} \left( p_{i_1}\dots p_{i_{k/d}} \right) ^d
\leq \sum_{\underset{d | k}{d \geq 2}}  \norm{\pp}_d^k.
\end{multline*}
Recall now that the function that maps $r \in [1, +\infty]$ to $\norm{\pp}_r$ is decreasing, since $\pp$ is non degenerate, and that we have the bound $\norm{\pp}_{2} \leq \sqrt{\norm{\pp}_{\infty}}$. Thus,
$$\left| 1- k \sum_{\ii \in \widetilde{Q}_k} p_{\ii} \right|   \leq 2\sqrt{k} \norm{\pp}_2^k
 \leq 2\sqrt{k} \sqrt{\norm{\pp}_{\infty}}^k
\leq 2\sqrt{k} R^{k/2} = o(1)$$
which is the stated result.
\end{proof}

We can now prove Theorem \ref{Cyclecountnormality}. The main idea is to show that the only terms that are non negligible in $\kappa_r(K_n)$ are the $\kappa_r(D_{\ii})$. These terms are equivalent to $p_{\ii}$, with sum (over all $\ii \in \widetilde{Q}_k$) equivalent to $\frac{1}{k}$, which gives us $\log(n)$ when summing over $k$.

\begin{proof}[Proof of Theorem \ref{Cyclecountnormality}]
First of all, define $\displaystyle \widetilde{K}_n = \sum_{k=1}^{n/ \log(\log(n))} C_k^{(n)}$, and note that $$ K_n - \widetilde{K}_n  = \sum_{k=n/ \log(\log(n))}^{n} C_k^{(n)} \leq \log(\log(n)) \text{ a.s.}$$ Hence, studying the convergence of $\frac{K_n - \log(n)}{\sqrt{\log(n)}}$ and $\frac{\widetilde{K}_n - \log(n)}{\sqrt{\log(n)}}$ are equivalent. We will therefore prove that, for any $r \geq 1, \kappa_r(\widetilde{K}_n) \sim \log(n)$, from which the stated result follows.\\
Let $r \geq 1$. We study the asymptotic behavior of the cumulant 
$$ \kappa_r(\widetilde{K}_n) = \sum_{k_1,...,k_r = 1}^{n/ \log(\log(n))} \sum_{\ii_1 \in \widetilde{Q}_{k_1}} \dots \sum_{\ii_r \in \widetilde{Q}_{k_r}} \kappa \left( D_{\ii_1},...,D_{\ii_r}  \right) . $$
Let $1 \leq m \leq r$, let $k_1,...,k_m \leq n/ \log(\log(n))$, let $\ii_1 \in \widetilde{Q}_{k_1},...,\ii_m \in \widetilde{Q}_{k_m}$ be pairwise distinct, and let $q_1,...,q_m \geq 1$ be such that $\sum q_j = r$. We study the cumulant $$ \kappa\left( \underbrace{D_{\ii_1}, \dots D_{\ii_1}}_{q_1\text{ times}} , \dots , \underbrace{D_{\ii_j}, \dots D_{\ii_j}}_{q_j\text{ times}} , \dots,  \underbrace{D_{\ii_m}, \dots D_{\ii_m}}_{q_m\text{ times}}   \right). $$
\textit{Case 1:} $m=1$. We write $ \ii_1 = \ii$ and $k_1 = k$. Then, the cumulant is reduced to $ \kappa_r \left( D_{\ii}  \right) $. By Proposition \ref{cumulantformula}, Theorem \ref{loiDiCycles} and Lemma \ref{momentsFctQueue},
\begin{align}
\kappa_r \left( D_{\ii} \right) & = \sum_{\pi} (|\pi|-1)! (-1)^{|\pi|-1} \prod_{B \in \pi} \E \left( D_{\ii}^{|B|} \right) \nonumber\\
 &=  \sum_{\pi} (|\pi|-1)! (-1)^{|\pi|-1} \prod_{B \in \pi} \sum_{l = 1}^{\lfloor n/k \rfloor} \left( l^{|B|} - (l-1)^{|B|} \right)p_{\ii}^l \nonumber\\
 & = \sum_{\pi} (|\pi|-1)! (-1)^{|\pi|-1} p_{\ii}^{|\pi|} \prod_{B \in \pi} \sum_{l = 0}^{\lfloor n/k \rfloor -1} \left( (l+1)^{|B|} - l^{|B|} \right)p_{\ii}^l \label{cumulantDi}
\end{align}
For a fixed nonempty $B \subset [r]$, and for $l$ large enough (only depending on $r$ and $R$), it holds $\left((l+1)^{|B|} - l^{|B|}\right) R^l \leq 2l^{|B|} R^l \leq 2l^r R^l \leq \left( \frac{R+1}{2} \right) ^l$. Hence, for $n$ large enough (only depending on $r$ and $R$),
$$\left| \sum_{l\geq \lfloor n/k \rfloor} \left( (l+1)^{|B|} - l^{|B|} \right) p_{\ii}^l \right| \leq \sum_{l \geq \log(\log(n)) } \left( \frac{R+1}{2} \right)^l = O\left( \left(\frac{R+1}{2} \right)^{\log(\log(n))} \right).$$
We thus have $\sum_{l = 0}^{\lfloor n/k \rfloor -1} \left( (l+1)^{|B|} - l^{|B|} \right)p_{\ii}^l = \sum_{l \geq 0} \left( (l+1)^{|B|} - l^{|B|} \right)p_{\ii}^l + O\left( \left(\frac{R+1}{2} \right)^{\log(\log(n))} \right)$.
Since the power series $\sum_{l\geq 0} \left( (l+1)^{|B|} - l^{|B|} \right) x^l$ has radius of convergence 1, and since we have $p_{\ii} \leq R^k \underset{k \rightarrow +\infty}{\rightarrow} 0$, then $\sum_{l \geq 0} \left( (l+1)^{|B|} - l^{|B|} \right)p_{\ii}^l \underset{k \rightarrow +\infty}{\sim} 1$. Recall that $S(r,s)$ is the number of partitions of $[r]$ in $s$ unordered subsets. Then, by eq.~(\ref{cumulantDi}),
\begin{align*}
\kappa_r \left( D_{\ii} \right) & \underset{k \rightarrow +\infty}{\sim} \left(1+ O\left( \left(\frac{R+1}{2} \right)^{\log(\log(n))} \right) \right) \left( \sum_{s=1}^r S(r,s)(s-1)!(-1)^{s-1} p_{\ii}^s \right) \\
& \underset{k \rightarrow +\infty}{\sim} \left(1+ O\left( \left(\frac{R+1}{2} \right)^{\log(\log(n))} \right) \right)  p_{\ii}
\end{align*}
where the $O$ only depends on $r$ and $R$. Now, by summing over all $\ii \in \widetilde{Q}_{k}$, and then over all $k \leq \frac{n}{\log(\log(n))}$, we get $$\kappa_r(D_{\ii}) \underset{n \rightarrow +\infty}{\sim} \log\left(\frac{n}{\log(\log(n))} \right)  \underset{n \rightarrow +\infty}{\sim} \log(n).$$\\
\textit{Case 2:} $m \geq 2$.
We prove that $$ \kappa\left( \underbrace{D_{\ii_1}, \dots D_{\ii_1}}_{q_1\text{ times}} , \dots , \underbrace{D_{\ii_j}, \dots D_{\ii_j}}_{q_j\text{ times}} , \dots,  \underbrace{D_{\ii_m}, \dots D_{\ii_m}}_{q_m\text{ times}}   \right) =  \left( \prod_{j=1}^m \frac{p_{\ii_j}}{k_j} \right)O\left(R^{n/2}\right) $$ where $O$ only depends on $r$ and $R$. We prove this result by induction on $r \geq 2$. Firstly, note that the case $r = 2$ is handled by Lemma \ref{powercumulants}. Assume now that the result holds for any $r' <r$ and any $m' \leq r'$.
Let $\pi_0$ be the partition of $[r]$ induced by $\ii_1, \dots, \ii_r$, i.e.~the partition with $m$ blocks defined by 
\begin{center}
$\pi_0 = \left\{ B_1,...,B_m \right\}$, where $B_j = \llbracket q_1+\dots+q_{j-1}+ 1,q_1+\dots+q_j  \rrbracket$.
\end{center}
By Theorem \ref{LeonovShiryaev}, we have
$$\kappa\left( D_{\ii_1}^{q_1},\dots,D_{\ii_m}^{q_m} \right) = \sum_{\pi \vee \pi_0 = \{ [r] \} } \prod_{B \in \pi} \kappa \left( D_{\ii_l}, l \in B \right)$$
By Lemma \ref{powercumulants}, the left side satisfies $$\kappa\left( D_{\ii_1}^{q_1},\dots,D_{\ii_m}^{q_m} \right) = \left( \prod_{j=1}^m \frac{p_{\ii_j}}{k_j} \right)O\left(R^{n/2}\right).$$ 
In the right sum, one term (for the partition $\pi = \{ [r] \}$) is the cumulant of interest, i.e. $$ \kappa\left( \underbrace{D_{\ii_1}, \dots D_{\ii_1}}_{q_1\text{ times}} , \dots , \underbrace{D_{\ii_j}, \dots D_{\ii_j}}_{q_j\text{ times}} , \dots,  \underbrace{D_{\ii_m}, \dots D_{\ii_m}}_{q_m\text{ times}}   \right).$$ 
Let $\pi \neq \{ [r] \}$ be a partition of $[r]$ satisfying $\pi \vee \pi_0 = \{ [r] \}$. We want to bound $\prod_{B \in \pi} \kappa \left( D_{\ii_l}, l \in B \right)$. If a block $B$ of $\pi$ is included in a block $B_j$ of $\pi_0$, then $\kappa \left( D_{\ii_l}, l \in B \right) \sim p_{\ii_j} = O(1)$ by case 1. Else, we can use the induction hypothesis, which gives us 
$$ \kappa \left( D_{\ii_l}, l \in B \right) = \left( \prod_{j, B \cap B_j \neq \emptyset}  \frac{p_{\ii_j}}{k_j} \right) O\left( R^{n/2} \right) .  $$
Now, since $\pi \vee \pi_0 = \{ [r] \}$, for any block $B_j$ of $\pi_0$,  there exists a block $B$ of $\pi$ such that $B \cap B_j \neq \emptyset$ and $B \not\subset B_j$. Thus, each $\frac{p_{\ii_j}}{k_j}$ appears at least once in the bound over $\prod_{B \in \pi} \kappa \left( D_{\ii_l}, l \in B \right)$. Hence, the result holds:
$$ \kappa\left( \underbrace{D_{\ii_1}, \dots D_{\ii_1}}_{q_1\text{ times}} , \dots , \underbrace{D_{\ii_j}, \dots D_{\ii_j}}_{q_j\text{ times}} , \dots,  \underbrace{D_{\ii_m}, \dots D_{\ii_m}}_{q_m\text{ times}}   \right) =  \left( \prod_{j=1}^m \frac{p_{\ii_j}}{k_j} \right)O\left(R^{n/2}\right) $$ where O only depends on $r$ and $R$.

To get the announced result, we now sum over all the $p_{\ii_j} \in \widetilde{Q}_{k_j}$ for $k_j  \leq n/ \log(\log(n))$. Hence,
\begin{align*}
\kappa_r(\widetilde{K}_n) & = \sum_{k_1,...,k_r = 1}^{n/ \log(\log(n))} \sum_{\ii_1 \in \widetilde{Q}_{k_1}} \dots \sum_{\ii_r \in \widetilde{Q}_{k_r}} \kappa \left( D_{\ii_1},...,D_{\ii_r}  \right)\\
& = \sum_{m=1}^r   \sum_{\left\{ A_1,\dots A_m   \right\} }  \sum_{k_1, \dots , k_m = 1}^{n/ \log(\log(n))} \sum_{\underset{\text{pairwise distinct}}{\ii_j \in \widetilde{Q}_{k_j}, j \leq m}}  \kappa \left( D_{\ii_1}, \dots ,D_{\ii_1}, \dots , D_{\ii_m}, \dots , D_{\ii_m} \right)\\
\end{align*}
where we split the sum with respect to the partition induced by the $\ii_j$, and where $D_{\ii_j}$ appears $\#A_j$ times in the cumulant. Now, we handle the term $m=1$ separately from the terms $m \geq 2$ to apply the above results: 
\begin{multline}\label{cumulantr}
\kappa_r(\widetilde{K}_n) = \left(1+ O\left( \left(\frac{R+1}{2} \right)^{\log(\log(n))} \right) \right) \sum_{k = 1}^{n/ \log(\log(n))} \sum_{\ii \in \widetilde{Q}_k}  p_{\ii}\\ + O \left( R^{n/2} \right) \sum_{m=2}^r   S(r,m) \sum_{k_1, \dots , k_m = 1}^{n/ \log(\log(n))} \sum_{\underset{\text{pairwise distinct}}{\ii_j \in \widetilde{Q}_{k_j}, j \leq m}}  \left( \prod_{j=1}^m \frac{p_{\ii_j}}{k_j} \right).
\end{multline}
From Lemma \ref{PrimNegl2}, it follows that $\sum_{\ii \in \widetilde{Q}_k}  p_{\ii} \underset{k \rightarrow + \infty}{\sim} \frac{1}{k}$, and then the first term is equivalent to $\log(n)$. In the other sums in $m$, we have the bound 
\begin{align*}
\sum_{k_1, \dots , k_m = 1}^{n/ \log(\log(n))} \sum_{\underset{\text{pairwise distinct}}{\ii_j \in \widetilde{Q}_{k_j}, j \leq m}}  \left( \prod_{j=1}^m \frac{p_{\ii_j}}{k_j} \right) &\leq  \sum_{k_1, \dots , k_m = 1}^{n/ \log(\log(n))} \sum_{\ii_j \in \widetilde{Q}_{k_j}, j \leq m}  \left( \prod_{j=1}^m \frac{p_{\ii_j}}{k_j} \right) \\ & \underset{n \rightarrow + \infty}{\sim}  \displaystyle \sum_{k_1, \dots , k_m = 1}^{n/ \log(\log(n))} \left( \prod_{j=1}^m \frac{1}{k_j^2} \right) = O(1)
\end{align*}
and thus the right term of eq.~(\ref{cumulantr}) is a $O(R^{n/2})$, which is negligible compared to $\log(n)$. Finally, it holds $$ \kappa_r(\widetilde{K}_n) \sim \log(n) $$ which proves the stated convergence, using  Proposition \ref{cumulantformula}. \\
We now show the convergence of the moments of $\frac{K_n - \log(n)}{\sqrt{\log(n)}}$. Define $$\delta_n = K_n - \widetilde{K}_n, \qquad A_n = \frac{\widetilde{K}_n - \log(n)}{\sqrt{\log(n)}}, \qquad B_n = \frac{\delta_n}{\sqrt{\log(n)}}.$$
Note that $\displaystyle \frac{K_n - \log(n)}{\sqrt{\log(n)}} = A_n + B_n$ and $B_n \leq \dfrac{\log(\log(n))}{\log(n)}$ a.s. 
Fix $r \geq 2$ which we suppose to be an even number. Then, note that for any $x,y \in \R, (x+y)^r  \leq 2^{r-1}(x^r + y^r)$ (since $x \mapsto x^r$ is a convex function). Hence, $$ \E \left( \left(A_n + B_n \right)^r  \right) \leq 2^{r-1} \left( \E(A_n^r) + \E(B_n^r) \right).$$
Since $B_n \leq \dfrac{\log(\log(n))}{\log(n)}$ a.s., its $r$th moment converges to 0. Moreover, we proved above that $\E(A_n^r)$ converges, thus it is bounded. Therefore, we proved that the even (absolute) moments of $\displaystyle \frac{K_n - \log(n)}{\sqrt{\log(n)}}$ are bounded, hence the convergence of all the moments follows from Proposition \ref{CVmoments}.
\end{proof}

\section*{Acknowledgments}

The author would like to express his sincere gratitude to his PhD supervisor Valentin Féray for his invaluable guidance and support.

The author would like to thank the ORION program for its contribution to the funding of AG's research internship. This work has benefited from a French government grant managed by the Agence Nationale de la Recherche with the reference ANR-20-SFRI-0009.

The author acknowledges partial support from French Research Agency (ANR), through the LOUCCOUM project ANR-24-CE40-7809.

\bibliography{bibliographie}
\bibliographystyle{bibli_perso}

\end{document}